\documentclass[a4paper,reqno]{amsart}
\usepackage[english]{babel}
\usepackage{amsmath, amssymb, amsthm, amscd}
\usepackage{enumerate}
\usepackage{palatino}
\usepackage{mathpazo}
\usepackage{mathrsfs}
\usepackage{paralist}
\usepackage[a4paper]{geometry}
\usepackage{url}
\usepackage{hyperref}
\usepackage[all]{xy}
\usepackage{stmaryrd}
\usepackage{graphicx}

\DeclareMathOperator{\id}{id}
\DeclareMathOperator{\im}{im}

\newcommand{\NN}{\mathbb{N}}
\newcommand{\ZZ}{\mathbb{Z}}
\newcommand{\RR}{\mathbb{R}}
\newcommand{\CC}{\mathbb{C}}
\newcommand{\CP}{\mathbb{C}P}

\newcommand{\Fix}{\mathrm{Fix}}
\newcommand{\diag}{\mathrm{diag}}
\newcommand{\Hor}{\mathrm{Hor}}
\newcommand{\Ad}{\mathrm{Ad}}
\newcommand{\Cut}{\mathrm{Cut}}
\newcommand{\Cuttil}{\widetilde{\Cut}}

\newcommand{\Isom}{\mathrm{Isom}}
\newcommand{\ev}{\mathrm{ev}}

\newcommand{\D}{\mathcal{D}}
\newcommand{\LL}{\mathcal{L}}
\newcommand{\Sc}{\mathcal{S}}

\newcommand{\Bhat}{\widehat{B}}
\renewcommand{\gg}{\mathfrak{g}}
\newcommand{\hh}{\mathfrak{h}}
\newcommand{\mm}{\mathfrak{m}}
\renewcommand{\aa}{\mathfrak{a}}
\newcommand{\kk}{\mathfrak{k}}
\newcommand{\dd}{\mathrm{d}}

\DeclareMathOperator{\GC}{\mathsf{GC}}
\DeclareMathOperator{\TC}{\mathsf{TC}}
\DeclareMathOperator{\cat}{\mathsf{cat}}
\DeclareMathOperator{\secat}{\mathsf{secat}}

\DeclareMathOperator{\Exp}{Exp}
\DeclareMathOperator{\rk}{rk}

\renewcommand{\setminus}{\smallsetminus}

\theoremstyle{plain}
\newtheorem{theorem}{Theorem}[section]
\newtheorem{prop}[theorem]{Proposition}
\newtheorem{lemma}[theorem]{Lemma}
\newtheorem{cor}[theorem]{Corollary}  
\newtheorem*{thm}{Theorem}

\theoremstyle{definition}
\newtheorem{definition}[theorem]{Definition}

\theoremstyle{remark}
\newtheorem{remark}[theorem]{Remark}
\newtheorem{example}[theorem]{Example}

\numberwithin{equation}{section}

\begin{document}
\setlength{\parindent}{0cm}

\title{Geodesic complexity of homogeneous Riemannian manifolds}
\author{Stephan Mescher}
\address{Institut f\"ur Mathematik \\ Martin-Luther-Universit\"at Halle-Wittenberg \\ Theodor-Lieser-Strasse 5 \\ 06120 Halle (Saale) \\ Germany}
\email{stephan.mescher@mathematik.uni-halle.de}

\author{Maximilian Stegemeyer}
\address{Max Planck Institute for Mathematics in the Sciences, Inselstrasse 22, 04103 Leipzig, Germany} 
\address{Mathematisches Institut, Universit\"at Leipzig, Augustusplatz 10, 04109 Leipzig, Germany}
\email{maximilian.stegemeyer@mis.mpg.de}
\date{\today}

\begin{abstract}
We study the geodesic motion planning problem for complete Riemannian manifolds and investigate their geodesic complexity, an integer-valued isometry invariant introduced by D. Recio-Mitter. Using methods from Riemannian geometry, we establish new lower and upper bounds on geodesic complexity and compute its value for certain classes of examples with a focus on homogeneous Riemannian manifolds. Methodically, we study properties of stratifications of cut loci and use results on their structures for certain homogeneous manifolds obtained by T. Sakai and others.
\end{abstract}
\maketitle

\setcounter{tocdepth}{1} 
\tableofcontents

\section{Introduction}

A topological abstraction of the motion planning problem in robotics was introduced by M. Farber in \cite{FarberTC}. The \emph{topological complexity} of a path-connected space $X$ is denoted by $\TC(X)$ and intuitively given by the minimal number of open sets needed to cover $X \times X$, such that on each of the open sets there exists a continuous motion planner. Here, a continuous motion planner is a map associating with each pair of points a continuous path from the first point to the second point, which varies continuously with the endpoints. Such maps are interpreted as algorithms telling an autonomous robot in the workspace $X$ how it is supposed to move from its position to a desired endpoint.  Unfortunately, the topological complexity of a space does not tell us anything about the feasibility or efficiency of the paths taken by motion planners having $\TC(X)$ domains of continuity, see the discussion in \cite[Introduction]{BlaszCarras}. For example, the explicitly constructed motion planners for configuration spaces of Euclidean spaces by H. Mas-Ku and E. Torres-Giese in \cite{MasTorres} and by Farber in \cite[Section 8]{FarberConf} require few domains of continuity, but have paths among their values which are far from being length-minimizing. Considering a general metric space, paths taken by the motion planners might become arbitrarily long and thus be unsuited for practical motion planning problems. 

Recently, D. Recio-Mitter has introduced the notion of \emph{geodesic complexity} of metric spaces in \cite{RecioMitter}. There, the paths taken by motion planners are additionally required to be length-minimizing between their endpoints. Intuitively, this is seen as the complexity of \emph{efficient} motion planning in metric spaces. Recio-Mitter's seminal article has already triggered research in geodesic complexity, especially  computations of geodesic complexity for interesting classes of examples, see \cite{DavisGeodConf}, \cite{DHRM} and \cite{DRM}.

In this article we study the geodesic complexity of complete Riemannian manifolds and derive new lower and upper bounds for their geodesic complexities by methods from Riemannian geometry. 

Before continuing, we recall the definition of geodesic complexity of geodesic spaces from \cite[Definition 1.7]{RecioMitter} for the special case of a complete Riemannian manifold. Let $(M,g)$ be a complete connected Riemannian manifold and let $PM:=C^0([0,1],M)$ be equipped with the compact-open topology. We recall that a geodesic segment $\gamma:[0,1] \to M$ is called \emph{minimal} if it minimizes the length compared to all rectifiable paths from $\gamma(0)$ to $\gamma(1)$. For simplicity, we shall call a minimal geodesic segment simply a \emph{minimal geodesic}. Consider
$$GM := \{ \gamma \in PM \ | \ \gamma \text{ is a minimal geodesic in }  (M,g)\}$$ 
as a subspace of $PM$ and let $$\pi: GM \to M \times M, \quad \pi(\gamma)=(\gamma(0),\gamma(1)).$$ By standard results from Riemannian geometry, $\pi$ is surjective since $(M,g)$ is complete, see \cite[Corollary 5.8.5]{Petersen}. The \emph{geodesic complexity of $(M,g)$} is given by $\GC(M,g)=r$, where $r \in \NN$ is the smallest integer with the following property: there are $r$ pairwise disjoint locally compact subsets $E_1,\dots,E_r \subset M\times M$ with $\bigcup_{i=1}^r E_i=M \times M$, such that for each $i \in \{1,2,\dots,r\}$  there exists a continuous geodesic motion planner $s_i:E_i \to GM$, i.e. a continuous local section of the map $\pi$. 
If there is no such $r$, we let $\GC(M,g)=+\infty$. Since it is not at all evident how to compute this number explicitly, one is interested in establishing lower and upper bounds for $\GC(M,g)$. This approach is also common in studies of Lusternik-Schnirelmann category or, more generally, sectional categories of fibrations. Given a fibration $p:E \to B$ the \emph{sectional category of $p$} is given by $\secat(p)=k$, where $k \in \NN$ is the minimal number  with the following property: there exists an open cover of $B$ consisting of $k$ open subsets, such that $p$ admits a continuous local section over each of these sets. This notion was introduced under the name \emph{genus of a fibration} by A. Schwarz in \cite{SchwarzGenus}. The topological complexity of a topological space $X$ is for example given as the sectional category of the fibration 
$$PX \to X \times X, \quad \gamma \mapsto (\gamma(0),\gamma(1)).$$Schwarz worked out several ways of obtaining lower and upper bounds for sectional categories which have direct consequences for topological complexity, see e.g. \cite{FarberSurveyTC} or \cite[Chapter 4]{FarberBook} for an overview. 

However, the restriction $\pi:GM \to M \times M$ of this fibration to minimal geodesics is in general \emph{not} a fibration. For example if $M=S^n$ is an $n$-sphere, where $n \in \NN$, and equipped with a round metric, then $\pi^{-1}(\{(p,q)\})$ consists of one element if $q \neq -p$, while it is homeomorphic to $S^{n-1}$ if $q=-p$. In particular, not all preimages are homotopy-equivalent, so $\pi$ is not a fibration in this case. Therefore, Schwarz's results are not applicable to the setting of geodesic complexity. Instead we will derive several lower and upper bounds for the geodesic complexity of Riemannian manifolds using methods from Riemannian geometry.  By \cite[Remark 1.9]{RecioMitter},  every complete Riemannian manifold satisfies $\TC(M) \leq \GC(M,g)$. This formalizes the observation that requiring the paths a robot takes to be as short as possible can increase the complexity of the problem. For example, as shown in \cite[Theorem 1.11]{RecioMitter}, for each $n \geq 3$ there exists a Riemannian metric $g_n$ on the sphere $S^n$ for which $\GC(S^n,g_n)-\TC(S^n) \geq n-3$. In practical applications, a person designing robotic systems that are supposed to move autonomously might not mind a higher complexity. In fact, such a person might accept more instabilities in the motions of robots as a downside if the upside is that the robots move fast and efficiently. 

An important observation is that the difficulties of geodesic motion planning lie in the cut loci of $(M,g)$, as was pointed out by Recio-Mitter in the more general framework of metric spaces in \cite[p. 144]{RecioMitter}. Let $\Cut_p(M)$ denote the cut locus of $p\in M$ in $(M,g)$. We refer to \cite[p. 308]{LeeRiemann}, \cite[p. 219]{Petersen} or Definition \ref{DefCutLocus} below for its definition. If $A \subset M \times M$ satisfies $q \notin \Cut_p(M)$ for each $(p,q) \in A$, then there is a unique minimal geodesic from $p$ to $q$ for each $(p,q) \in A$. The corresponding geodesic motion planner $A \to GM$ is continuous, see also the observations of Z. B\l{}aszczyk and J. Carrasquel Vera from \cite{BlaszCarras}. Thus, to compute the geodesic complexity of a manifold, we need to understand its cut loci.  While the cut locus of a point in a Riemannian manifold is always closed and of measure zero, see \cite[Theorem 10.34.(a)]{LeeRiemann}, little else is known about cut loci in general. 

In \cite[Corollary 3.14]{RecioMitter}, Recio-Mitter establishes a lower bound on the geodesic complexity of metric spaces given in terms of the structure of their cut loci. He considers cut loci which possess stratifications admitting finite coverings. For this purpose, Recio-Mitter introduces the notion of a \emph{level-wise stratified covering} in \cite[Definition 3.8]{RecioMitter}. He then defines a notion of \emph{inconsistency}, which is roughly a condition on the relations between the coverings of the different strata of cut loci by minimal geodesics. It formalizes certain incompatibility properties of families of geodesics connecting a point with points in its cut locus.

Focusing on complete Riemannian manifolds, we will use Riemannian exponential maps to establish a similar inconsistency condition on cut loci, which is more concise than the one from \cite{RecioMitter}. Given a complete Riemannian manifold $M$ and a point $p \in M$ for which $\Cut_p(M)$ admits a stratification, we study the preimages of the different strata of $\Cut_p(M)$ under the Riemannian exponential map $\exp_p:T_pM \to M$. Assume that some $x \in M$ lies in the closure of multiple connected components of the same stratum of $\Cut_p(M)$. We then study the closures of the preimages of all of these components under $\exp_p$ as subsets of $T_pM$. The inconsistency condition demands that these closures have no point in common that is mapped to $x$ by $\exp_p$. We will see that this condition excludes the existence of an open neighborhood $U$ of $x$ with a single continuous geodesic motion planner which connects $p$ to all points of $\Cut_p(M)$ that lie in $U$. 

Note that our definition is only applicable to Riemannian manifolds and not to arbitrary geodesic spaces. One of its benefits in the Riemannian setting is the fact that we can deduce an easier condition than the one introduced by Recio-Mitter. More precisely, we do not require anymore that any point in a cut locus of another point is connected to that point by only finitely many minimal geodesics. Moreover, our inconsistency condition is explicitly stated as an intersection condition on certain subsets of a tangent cut locus instead of using the notion of level-wise stratified coverings as in \cite{RecioMitter}. 

Our main result on inconsistent stratifications is the following theorem.  This result is similar to \cite[Corollary 3.14]{RecioMitter} and our proof is inspired by Recio-Mitter's proof as well.
\begin{thm}[Theorem \ref{TheoremLowerBound}]
Let $(M,g)$ be a closed Riemannian manifold. Assume that there exists a point $p \in M$ for which $\Cut_p(M)$ admits an inconsistent stratification of depth $N \in \NN$. Then
$$\GC(M) \geq  N+1.$$
\end{thm}

There is more to say about cut loci of \emph{homogeneous} Riemannian manifolds, i.e. Riemannian manifolds $(M,g)$ whose isometry groups act transitively on $M$. An isometry $\phi: M \to M$ maps the cut locus of $p \in M$ onto the one of $\phi(p)$. Hence, the cut locus of a point is identified with the one of another point by an isometry. This translation property of the cut loci allows us to estimate the geodesic complexity of $M$ from above, once we understand how we can decompose one single cut locus into domains of continuous geodesic motion planners. The following result provides an upper bound for geodesic complexity in terms of a sectional category and the subspace geodesic complexities of considerably smaller subsets of $M \times M$. Here, the subspace geodesic complexity of $A \subset M \times M$ is defined in terms of covers of $A$ by domains of continuous geodesic motion planners.

\begin{thm}[Corollary \ref{CorUpperStrat}]
Let $(M,g)$ be a homogeneous Riemannian manifold and let $\Isom(M,g)$ denote its isometry group. Let $p \in M$ and assume that $\Cut_p(M)$ has a stratification $(S_1,\dots,S_k)$ of depth $k$. Then 
$$\GC(M) \leq \secat(\ev_p:\Isom(M,g)\to M)\cdot  \sum_{i=1}^k \max_{Z_i \in \pi_0(S_i)} \GC_p(Z_i)+1,$$
where $\ev_p(\phi) = \phi(p)$ for all $\phi \in \Isom(M,g)$ and where $\GC_p(Z_i)$ is the subspace geodesic complexity of $\{p\} \times Z_i \subset M \times M$.
\end{thm}

In the case of compact simply connected irreducible symmetric spaces, we are able to further estimate this upper bound from above in terms of certain sectional categories. This means that for such symmetric spaces we obtain an upper bound on $\GC(M)$ which does not involve any geodesic complexities.

Note that this result produces the first upper bound for geodesic complexity in terms of categorical invariants. Indeed, the only previously known upper bounds were derived by Recio-Mitter in \cite{RecioMitter} either from explicit constructions of geodesic motion planners or from the existence of particularly simple coverings of cut loci. We pick up Recio-Mitter's so-called trivially covered stratifications in this article in the setting of Riemannian manifolds as well. 

In addition to establishing new lower and upper bounds for geodesic complexity, we compute the geodesic complexities of some Riemannian manifolds whose cut loci are well-understood. We will show that every three-dimensional Berger sphere $(S^3,g_\alpha)$ satisfies $\GC(S^3,g_\alpha)=2$
and that $\GC(T^2,g_f)=3$ for every flat metric $g_f$ on the two-dimensional torus. This extends the two-dimensional case of Recio-Mitter's computation of the geodesic complexity of the standard flat $n$-torus from \cite[Theorem 4.4]{RecioMitter}.

\bigskip 

The article is structured as follows: In Section \ref{sec2} we introduce some additional terminology and recall elementary facts about geodesic complexity and cut loci. Section \ref{sec3} contains some basic non-existence results on continuous geodesic motion planners. These results illustrate the difficulties for motion planning that cut loci can create. In Section \ref{sec4} we establish lower bounds on geodesic complexity by two different approaches. On the one hand this is done in terms of principal bundles over the manifold and the topological complexities of their total spaces. On the other hand we study manifolds with stratified cut loci whose stratifications satisfy the abovementioned inconsistency property. We focus on homogeneous Riemannian manifolds in Section \ref{sec5}. More precisely, we show that their geodesic complexities can be estimated from above in terms of the subspace complexities of a single cut locus. In Section \ref{sec6} we consider Riemannian manifolds whose cut loci admit trivially covered stratifications. For such stratifications the relations between a cut locus and its corresponding tangent cut locus are particularly simple. Section \ref{sec7} deals with examples of geodesic complexities. Combining results from the previous sections with new observations, we re-obtain Recio-Mitter's computation of geodesic complexity of the standard flat $n$-torus and determine the geodesic complexity of arbitrary flat $2$-tori. As another class of examples, we explicitly compute the geodesic complexity of three-dimensional Berger spheres. 
 In the final Section \ref{sec8} we consider consequences of the previous results for compact simply connected symmetric spaces. In both situations, the considered cut loci have been studied by T. Sakai. Using the estimates from Section \ref{sec5}, we derive an upper bound for geodesic complexity that is given in terms of the Lie groups from which the symmetric space is built. We further make explicit computations for two examples of symmetric spaces.

\subsection*{Acknowledgements}

The authors thank the anonymous referee for their careful and thoughtful reading of our manuscript. Their suggestions highly  improved the exposition and the clarity of the article. 
\bigskip 

\emph{Throughout this article we assume all manifolds to be smooth and connected and all Riemannian metrics to be smooth.} 
\section{Basic notions and definitions}
\label{sec2}
We begin this article by introducing subspace versions of geodesic complexity for Riemannian manifolds. Afterwards, we recall some basic computations from \cite{RecioMitter} and several facts about cut loci in Riemannian manifolds.

\begin{definition}
Let $(M,g)$ be a complete Riemannian manifold and let $\pi:GM \to M \times M$, $\pi(\gamma)=(\gamma(0),\gamma(1))$. Let $GM$ be equipped with the subspace topology of $C^0([0,1],M)$ with the compact-open topology.
\begin{enumerate}[a)]
\item Let $X \subset M \times M$. A \emph{geodesic motion planner on $X$} is a section $s:X \to GM$ of $\pi$.
\item Given $A \subset M \times M$ we let $\GC_{(M,g)}(A)$ be the minimum $r \in \NN$, for which there are $r$ pairwise disjoint locally compact subsets $E_1,\dots,E_r \subset M\times M$, such that  $A \subset \bigcup_{i=1}^r E_i$ and such that for each $i \in \{1,2,\dots,r\}$ there exists a continuous geodesic motion planner $s_i:E_i \to GM$. If no such $r$ exists, then we put $\GC_{(M,g)}(A):=+\infty$. We call $\GC_{(M,g)}(A)$ the \emph{subspace geodesic complexity of $A$}.
\end{enumerate}
\end{definition}
We recall that the map $\pi$ is surjective for complete Riemannian manifolds. This is a consequence of the Hopf-Rinow theorem, see \cite[Corollary 5.8.5]{Petersen}.

\begin{remark}
\label{RemarkGCelem}
\begin{enumerate}[(1)]
\item If it is obvious which Riemannian metric we are referring to, we occasionally suppress it from the notation and write $$\GC(M) := \GC(M,g) \qquad \text{and} \qquad  \GC_M(A):= \GC_{(M,g)}(A).$$ Note that in particular $\GC(M) = \GC_M(M \times M)$. 
\item Given $p \in M$ and $B \subset M$, we further put
$$\GC_p(B) := \GC_{(M,g)}(\{p\}\times B).$$
\item Our definition differs from Recio-Mitter's original definition by $1$ in the sense that for us $\GC(\{\text{pt}\})=1$, while it would be $0$ in the sense of \cite[Definition 1.7]{RecioMitter}.
\end{enumerate}
\end{remark}
\begin{example}
\label{ExampleGC}
\begin{enumerate}[(1)]
\item  As proven in \cite[Proposition 4.1]{RecioMitter}, if $g_r$ is a round metric on the sphere $S^n$, where $n \in \NN$, then 
$$\GC(S^n,g_r)=\TC(S^n) = \begin{cases}
2 & \text{if $n$ is odd,} \\
3 & \text{if $n$ is even.}
\end{cases}$$
\item Let $g_f$ be the standard flat metric on $T^2$ and let $g_{emb}$ denote the metric induced by the standard embedding $T^2 \hookrightarrow \RR^3$ and the Euclidean metric on $\RR^3$. By \cite[Theorems 4.4 and 5.1]{RecioMitter} it holds that 
$$\GC(T^2,g_f)=3, \qquad \GC(T^2,g_{emb})=4.$$ 
\item It was further shown in \cite[Theorem 1.11]{RecioMitter} that for each $k \in \NN$ with $k \geq 3$ there exists a Riemannian metric $g_k$ on $S^{k}$ with $\GC(S^{k},g_k) \geq k$.
\end{enumerate}
\end{example}

\begin{remark} 
\label{RemarkTC}
Let $(M,g)$ be a complete Riemannian manifold.
\begin{enumerate}[(1)]
\item For all $A\subset M \times M$ it holds that $\TC_M(A) \leq \GC_M(A)$, where $\TC_M(A)$ is the relative topological complexity of $A$ in $M \times M$, see \cite[Section 4.3]{FarberBook}. Here, we made use of the characterization of topological complexity by locally compact subsets shown in \cite[Proposition 4.9]{FarberBook}.
\item It is easy to see that
\begin{equation}
\label{EqGCunion}
\GC_M(A\cup B) \leq \GC_M(A) + \GC_M(B) \qquad \forall A,B \subset M \times M.
\end{equation}
This is shown in analogy with \cite[Proposition 4.24]{FarberBook}.
\end{enumerate}
\end{remark}

As pointed out by Recio-Mitter, the crucial ingredients for the discussion of geodesic complexity are the cut loci of points in the space under consideration. The notions of cut loci in metric and in Riemannian geometry are slightly different from each other. While Recio-Mitter used the former notion in his work, see \cite[Definition 3.1]{RecioMitter}, we will use the latter throughout this manuscript. We will recall the notion of cut loci from Riemannian geometry in the following definition. The relation between the two will be discussed in Remark \ref{RemarkOCut}.(3) below. See also \cite[p. 308]{LeeRiemann} or \cite[p. 219]{Petersen} for the following definition.

\begin{definition}
\label{DefCutLocus}
Let $(M,g)$ be a complete Riemannian manifold and let $p \in M$.
\begin{enumerate}[a)]
\item Let $\gamma :[0,+\infty) \to M$ be a unit-speed geodesic with $\gamma(0) = p$ and $\dot{\gamma}(0)\in T_p M$. The \emph{cut time of $\gamma$} is given by 
$$   t_{\mathrm{cut}}(\gamma) = \sup \{ t > 0\,|\, \gamma|_{[0,t]} \,\,\text{is minimal}\} .   $$
If $t_{\mathrm{cut}}(\gamma)$ is finite, then $t_{\mathrm{cut}}(\gamma)\Dot{\gamma}(0)\in T_p M$ is a \emph{tangent cut point of $p$} and $\gamma(t_\mathrm{cut}(\gamma))\in M$ is a \emph{cut point of $p$ along $\gamma$}. Note that $$\gamma(t_\mathrm{cut}(\gamma))= \exp_p(t_{\mathrm{cut}}(\gamma) \Dot{\gamma}(0)).$$
\item The set of all cut points of $p$ is called the \emph{cut locus of $p$} and denoted by $\Cut_p(M)$. The set of all tangent cut points of $p$ is called the \emph{tangent cut locus of $p$} and denoted by $\Cuttil_p(M)$.
\item The \emph{total cut locus of $M$} is given by
$$\Cut(M) := \bigcup_{p \in M} (\{p\} \times \Cut_p(M)) \subset M \times M. $$
\end{enumerate}
\end{definition}
\begin{example}
Let $n \in \NN$ and let $g$ be a round metric on the sphere $S^n$. Then, by \cite[Example 10.30.(a)]{LeeRiemann}, $\Cut_p(S^n)=\{-p\}$ for every $p \in S^n$. 
\end{example}
Further examples of cut loci will appear in the upcoming sections.

\begin{remark} 
\label{RemarkOCut}
Let $(M,g)$ be a complete Riemannian manifold.
\begin{enumerate}[(1)]
\item In general, $\Cut_p(M)$ does not need to be a submanifold of $M$. H. Gluck and D. Singer have shown in \cite[Theorem A]{glucksinger} that if $\dim M \geq 2$, then there exists a Riemannian metric on $M$ and a point $p\in M$ for which $\Cut_p(M)$is not triangulable.
\item By \cite[Theorem 3.3]{RecioMitter} there exists a continuous geodesic motion planner $$(M \times M) \setminus \Cut(M) \to GM,$$ from which Recio-Mitter derived that $\GC(M)=1$ if $\Cut(M)=\emptyset$. By \cite[Lemma 4.2]{BlaszCarras}, $(M\times M)\setminus \Cut(M)$ is open and therefore locally compact. Using \eqref{EqGCunion}, this shows that
$$\GC(M,g) \leq \GC_M(\Cut(M))+1.$$

\item Let $p\in M$. By \cite[p. 133]{Bishop}, the set of points  $q \in M$ such that there is more than one minimal geodesic from $p$ to $q$ is a dense subset of $\Cut_p(M)$.
This set is also called the \emph{ordinary cut locus} of $p$. In metric geometry, in particular in \cite[Definition 3.1]{RecioMitter}, the ordinary cut locus of a point is called its cut locus. The reader should thus keep in mind that the cut locus of a point as considered in \cite{RecioMitter}, is \emph{not} the cut locus of a point in the sense of this article, but a dense subset of the cut locus. 
\end{enumerate}
\end{remark}

\section{Non-existence results for geodesic motion planners}
\label{sec3}
We begin our study by discussing two non-existence results showing that certain subsets of a Riemannian manifold never admit continuous geodesic motion planners. First, we will study complete oriented Riemannian manifolds and see that the Euler class obstructs the existence of some geodesic motion planners. Then we will show that a complete Riemannian manifold $(M,g)$ has the following property: if a subset $A \subset M \times M$ contains an element of the total cut locus in its interior, then there will be no continuous geodesic motion planner on $A$. Before doing so, we first want to establish a technical proposition that we will make frequent use of throughout the article.

\begin{definition}
\label{DefVelocity}
Let $(M,g)$ be a complete Riemannian manifold. We call the map 
$$v: GM \to TM, \qquad v(\gamma)=\dot\gamma(0), $$
the \emph{velocity map of $GM$}.
\end{definition}

\begin{prop}
\label{PropVelocityCont}
Let $(M,g)$ be a complete Riemannian manifold. The velocity map $v: GM \to TM$ is continuous. 
\end{prop}

\begin{proof}
Let $(\gamma_n)_{n\in\mathbb{N}}$ be a convergent sequence in $GM$ and let $\gamma:= \lim_{n \to \infty}\gamma_n\in GM$. By our choice of topology on $GM$, this means that
\begin{equation}
\label{EqCOpoint}
\lim_{n \to \infty} \gamma_n(t)=\gamma(t) \quad  \forall t \in [0,1].
\end{equation}
We need to show that $\lim_{n \to \infty} v(\gamma_n)=v(\gamma)$. Let $\LL_g: GM\to \RR$ denote the length of a minimal geodesic with respect to $g$. From the minimality property of the curves,  we derive that
$$\lim_{n \to \infty} \LL_g(\gamma_n)= \lim_{n \to \infty} \dd_M(\gamma_n(0),\gamma_n(1))= \dd_M(\gamma(0),\gamma(1))=\LL_g(\gamma),$$
where $\dd_M: M \times M \to \RR$ is the distance function induced by $g$. Let $|\cdot|:TM \to \RR$ denote the fiberwise norm induced by $g$. Since $\LL_g(\alpha) =|\dot\alpha(0)|=|v(\alpha)|$ for each $\alpha \in GM$, it follows that
\begin{equation}
\label{EqNormsConv}
\lim_{n \to \infty} |v(\gamma_n)| = |v(\gamma)|.
\end{equation}
To show the continuity of $v$, we need to derive that $\lim_{n\to \infty} v(\gamma_n)=v(\gamma)$. Let $$\Exp:TM \to M \times M, \quad \Exp(p,v)=(p,\exp_p(v)),$$ 
be the extended exponential map. Let $K \subset M$ be a compact neighborhood of $\gamma(0)$ and let $$\rho_0:=\sup \{r > 0 \ | \ \exp_x|_{B_r(0)} \text{ is injective} \ \forall x\in K\},$$ where $B_r(0)$ denotes the open $n$-ball around the origin in the respective tangent space. Since $K$ is compact, $\rho_0>0$ by \cite[Lemma 6.16]{LeeRiemann}. For $r \in (0,\rho_0)$ we put 
$$D_r K := \{(p,v) \in TM \  | \ p \in K, \ \|v\|\leq r\},$$ 
i.e. $D_rK$ is the closed disk bundle over $K$ of radius $r$. Then $\Exp$ maps $D_rK$  diffeomorphically onto its image 
$$V_rK:= \Exp(D_rK)= \{(p,q) \in K \times M \ | \  \dd_M(p,q)\leq r\}.$$
Let $\Exp_K:D_rK \to V_rK$ denote the corresponding restriction of $\Exp$. Since $\Exp_K:D_rK \to V_rK$ is a diffeomorphism, its inverse $\Exp_K^{-1}:V_rK \to D_rK$ is a diffeomorphism as well. Thus, if we choose and fix a distance function $\dd_{TM}:TM \times TM \to \RR$ which induces the topology of $TM$, then  $\Exp^{-1}_K:V_rK \to D_rK$ is locally Lipschitz-continuous with respect to $\dd_M\times \dd_M$ and $\dd_{TM}$. We further observe that for all $\alpha \in GM$ with $\alpha(0)\in K$ and  $\dd_M(\alpha(0),\alpha(1)) \leq r$ it holds that 
$$\Exp^{-1}_K(\alpha(0),\alpha(1))=v(\alpha).$$
We consider two different cases: \medskip 

\textit{Case 1:} \enskip Assume that $|v(\gamma)|< r$. This implies that $(\gamma(0),\gamma(1))\in V_rK$. Then, by \eqref{EqNormsConv}, there exists $n_0\in \NN$ with 
$$\gamma_n(0)\in K \qquad \text{and} \qquad  |v(\gamma_n)|\leq r \qquad \forall n \geq n_0.$$ Thus, $(\gamma_n(0),\gamma_n(1))\in V_rK$ for all $n \geq n_0$. Let $C$ be a local Lipschitz constant for $\Exp^{-1}_K$ in a neighborhood of $(\gamma(0),\gamma(1))$. Then for sufficiently big $n \in \NN$ 
\begin{align*}
\dd_{TM}(v(\gamma_n),v(\gamma)) &= \dd_{TM}\left(\Exp^{-1}_K(\gamma_n(0),\gamma_n(1)),\Exp^{-1}_K(\gamma(0),\gamma(1))\right) \\
&\leq C (\dd_M(\gamma_n(0),\gamma(0))+\dd_M(\gamma_n(1),\gamma(1))).
\end{align*}
By \eqref{EqCOpoint}, this yields $\lim_{n \to \infty} \dd_{TM}(v(\gamma_n),v(\gamma))=0$, which we wanted to show.
\medskip 

\textit{Case 2:} \enskip Consider the case that $|v(\gamma)| \geq r$. By \eqref{EqNormsConv}, there exists $n_1 \in \NN$, such that 
$$\gamma_n(0) \in K \qquad  \text{and} \qquad  |v(\gamma_n)|<|v(\gamma)|+1 \qquad  \forall n \geq n_1.$$ 
Let $c := r/(|v(\gamma)|+1)  \in (0,1)$ and put $\xi_n := c\cdot v(\gamma_n)$ for each $n \in \NN$ and $\xi := c\cdot v(\gamma)$. Then $\xi \in D_rK$ and $\xi_n \in D_rK$ for all $n \geq n_1$. If we define 
$$\tilde\gamma_n,\tilde\gamma:[0,1] \to M, \qquad \tilde\gamma_n(t):=\gamma_n(ct),  \ \ \tilde\gamma(t):= \gamma(ct) \quad  \forall t \in [0,1],$$ 
then $\tilde\gamma,\tilde\gamma_n \in GM$ with $v(\tilde\gamma)=\xi$ and $v(\tilde\gamma_n)=\xi_n$ for each $n \geq n_1$. Since $(\gamma_n)_{n \in \NN}$ converges to $\gamma$ in the $C^0$-topology, it easily follows that $\lim_{n \to \infty} \tilde\gamma_n = \tilde\gamma$ in the $C^0$-topology as well. Thus, it follows from Case 1 that $\lim_{n \to \infty} \xi_n = \xi$, which obviously yields $\lim_{n \to \infty} v(\gamma_n)=v(\gamma)$. 
\end{proof}

In the following proposition, we observe that the Euler class of an oriented manifold can obstruct the existence of geodesic motion planners. 

\begin{prop}
\label{PropGraphGeodMP}
Let $(M,g)$ be a complete oriented Riemannian manifold whose Euler class is non-vanishing. Let $f:M \to M$ be a continuous map with $f(p)\neq p$ for all $p \in M$. If $A\subset M \times M$ satisfies $\mathrm{graph} \; f \subset A$, then there will be no continuous geodesic motion planner on $A$. 
\end{prop}
\begin{proof}
Assume by contradiction that there exists a continuous geodesic motion planner $s:A \to GM$. Then by Proposition \ref{PropVelocityCont} the map 
$$g: M \to TM, \qquad g(p) = (v\circ s)(p,f(p)),$$
is a continuous vector field, where $v$ is the velocity map. Since $f(p) \neq p$ for each $p$, the geodesic $s(p,f(p))$ is non-constant for all $p \in M$. Hence, $g(p)\neq 0$ for all $p \in M$. But such a vector field can not exist since the Euler class of $M$ is non-vanishing. This shows the claim. 
\end{proof}

\begin{cor}
\label{CorfmapstoOCut}
Let $(M,g)$ be a complete oriented manifold whose Euler class is non-vanishing. Let $f:M \to M$ be continuous and fixed-point free. Then for every Riemannian metric $g$ on $M$ there exists $p \in M$ with $f(p) \in \Cut_p(M,g)$.
\end{cor}
\begin{proof}
Assume by contradiction that there was such a metric $g$ for which $f(p) \notin \Cut_p(M,g)$ for all $p \in M$. Then $\mathrm{graph} \;f$ lies in $(M \times M) \setminus \Cut(M,g)$. But since there exists a continuous geodesic motion planner on $(M \times M) \setminus \Cut(M,g)$, see Remark \ref{RemarkOCut}.(3), this contradicts Proposition \ref{PropGraphGeodMP}. Hence, such a metric does not exist.
\end{proof}

\begin{cor}
Let $n \in \NN$. For every Riemannian metric $g$ on $S^{2n}$ there exists $p \in S^{2n}$, such that $-p \in \Cut_p(S^{2n},g)$.
\end{cor}
\begin{proof}
Apply Corollary \ref{CorfmapstoOCut} to the case of $M=S^{2n}$ and $f(x)=-x$.
\end{proof}

\begin{remark}
Our Corollary \ref{CorfmapstoOCut} is complementary to results of M. Frumosu and S. Rosenberg from \cite[p. 338]{FrumosuRose}. In this article the authors studied far point sets, i.e. sets of points mapped to their cut loci under self-maps of a Riemannian manifold, in a very general way. Frumosu and Rosenberg focused on self-maps whose far point sets are infinite and established connections to the Lefschetz numbers of such maps.
\end{remark}

In \cite[Remark 3.17]{RecioMitter}, Recio-Mitter mentioned that whenever a subset of $M \times M$ contains a point of the total cut locus in its interior, there is no continuous geodesic motion planner defined on that subset. For the sake of completeness, we report here a proof in the case of Riemannian manifolds.


\begin{prop}
\label{PropContMPCut}
Let $(M,g)$ be a complete Riemannian manifold, $p \in M$, $q \in \Cut_p(M)$ and let $U\subset M$ be an open neighborhood of $q$. Then there is no continuous geodesic motion planner on $\{p\} \times U$.
\end{prop}
\begin{proof}
As discussed in Remark \ref{RemarkOCut}.(3), the set of points $r\in M$ for which there is more than one minimal geodesic from $p$ to $r$ is dense in $\Cut_p(M)$. Hence, $U$ contains a point $q_0$ such that there are at least two minimal geodesics from $p$ to $q_0$.
In the following, we thus assume w.l.o.g. that $q$ itself has this property. Assume that a continuous geodesic motion planner $s:\{p\}\times U \to GM$ existed. By our choice of $q$, there are $\gamma_1,\gamma_2 \in GM$ with $$\gamma_1\neq \gamma_2, \qquad \gamma_1(0)=\gamma_2(0)=p \qquad \text{and} \qquad \gamma_1(1)=\gamma_2(1)=q.$$ Let $(t_n)_{n \in \NN}$ be a sequence in $(0,1)$ with $\lim_{n \to \infty} t_n=1$ and $\gamma_1(t_n),\gamma_2(t_n) \in U$ for all $n \in \NN$. One checks without difficulties that $\gamma_1(t) \neq \gamma_2(t)$ for all $t \in (0,1)$, such that in particular $\gamma_1(t_n)\neq\gamma_2(t_n)$ for all $n \in \NN$.

By definition of a cut locus, it follows for all $r \in (0,1)$ and $i \in \{1,2\}$ that 
$$\gamma_{i,r} \in GM, \quad \gamma_{i,r}(t):=\gamma_i(rt),$$ is the unique minimal geodesic from $p$ to $\gamma_i(r)$. In particular, this shows that necessarily
\begin{equation}
\label{Eqsgamma}
s(p,\gamma_i(t_n))= \gamma_{i,t_n} \qquad \forall n \in \NN, \ i \in \{1,2\}.
\end{equation}
Let $v:GM \to TM$ be the velocity map. It follows from Proposition \ref{PropVelocityCont} that $$v \circ s:\{p\} \times U \to TM$$ is continuous. Since $\gamma_1 \neq \gamma_2$, there are $\xi_1,\xi_2 \in T_pM$ with $\xi_1 \neq \xi_2$, such that $\gamma_1(t)=\exp_p(t\xi_1)$ and $\gamma_2(t)=\exp_p(t\xi_2)$ for all $t \in [0,1]$. By \eqref{Eqsgamma} and the fact that the differential of $\exp_p$ in $0$ is $\id_{T_pM}$, we thus obtain that
$$\lim_{n \to \infty} (v \circ s)(p,\gamma_i(t_n)) = \lim_{n \to \infty} \dot\gamma_{i,t_n}(0)= \lim_{n \to \infty} t_n \xi_i = \xi_i.$$
In particular, $\lim_{n \to \infty} (v\circ s)(p,\gamma_1(t_n))\neq \lim_{n \to \infty} (v\circ s)(p,\gamma_2(t_n))$. This contradicts the continuity of $s$, since by assumption $(v \circ s)(p,\gamma_1(1))=(v \circ s)(p,\gamma_2(1))$. Thus, such a continuous $s$ does not exist.
\end{proof}

The previous proposition has an immediate consequence in terms of geodesic complexity.
 
\begin{cor}
Let $(M,g)$ be a complete Riemannian manifold and let $A\subset M \times M$ be a locally compact subset with $$\mathrm{int}(A) \cap \Cut(M) \neq \emptyset,$$ 
where $\mathrm{int}(A)$ is the interior of $A$ as a subset of $M \times M$. Then $\GC_M(A) \geq 2$. 
\end{cor}
\begin{proof}
Assume that there was a continuous geodesic motion planner $s:A \to GM$. Let $(p,q) \in \mathrm{int}(A) \cap \Cut(M)$. By definition of the product topology, there are open neighborhoods $U$ of $p$ and $V$ of $q$ with $U \times V \subset \mathrm{int}(A)$, so in particular, $s|_{\{p\}\times V}$ would be a continuous geodesic motion planner. Since  $q \in \Cut_p(M)$, this contradicts Proposition \ref{PropContMPCut}, so there is no such motion planner. This shows that $\GC_M(A)\geq 2$. 
\end{proof}

\begin{remark}
There is another connection between cut loci and another numerical invariant, namely the Lusternik-Schnirelmann category of a Riemannian manifold $M$, which we denote by $\cat(M)$. Here, we use the convention that $\cat(X)=1$ if $X$ is contractible. One observes that $M \setminus \Cut_p(M)$ is contractible for all $p \in M$, which follows from \cite[Theorem 10.34.(c)]{LeeRiemann}. If $p_1,\dots,p_k \in M$ satisfy $\bigcap_{i=1}^k \Cut_{p_i}(M)=\emptyset$, then $$\{M \setminus \Cut_{p_1}(M),\dots,M \setminus \Cut_{p_k}(M)\}$$ will be an open cover of $M$ by contractible subsets, hence $\cat(M) \leq k$. By contraposition this shows that if $\cat(M) \geq k+1$ for some $k \in \NN$, then for every choice of $p_1,\dots,p_k\in M$ it holds that $$\bigcap_{i=1}^k \Cut_{p_i}(M) \neq \emptyset.$$
\end{remark}

\section{Lower bounds for geodesic complexity}
\label{sec4}
Lower bounds on topological complexity are mostly derived from the cohomology rings of a space. In this section, we derive lower bounds on geodesic complexity from the Riemannian structures of manifolds. We first establish a result involving a principal bundle over the manifold under consideration. By explicitly constructing motion planners, we will establish a lower bound on geodesic complexity in terms of categorical invariants of total space and fiber of the bundle.  Afterwards, we will establish the notion of inconsistent stratification that we lined out in the introduction of this article. Then we will go on to prove the second theorem stated in that introduction. 

We first establish a technical lemma whose proof follows the one of \cite[Theorem 13.1]{FarberSurveyTC}.

\begin{lemma}
\label{LemmaSecatDecomp}
Let $E$ and $X$ be topological spaces, let $p:E \to X$ be a fibration with $r:=\secat(p)<+\infty$ and assume that $X$ is normal. Then there are pairwise disjoint locally compact subsets $A_1,\dots,A_r \subset X$ with $X= \bigcup_{i=1}^r A_i$, such that for each $i \in \{1,2,\dots,r\}$ there exists a continuous local section $A_i \to E$ of $p$.
\end{lemma}
\begin{proof}
Let $\{U_1,\dots,U_r\}$ be an open cover of $X$, such that for each $i \in \{1,2,\dots,r\}$ there exists a continuous local section $s_i:U_i \to E$ of $p$. Since $X$ is normal, there exists a partition of unity $\{f_1,\dots,f_r\}$ subordinate to this finite open cover by \cite[Theorem 36.1]{Munkres}. Let $c_1,\dots,c_r \in (0,+\infty)$ with $c_1+\dots+c_r=1$. For each $i \in \{1,2,\dots,r\}$ we put
$$A_i := \{ x\in X \ | \ f_i(x) \geq c_i, \ f_j(x) < c_j \ \forall j < i\}.$$
Each $A_i$ is the intersection of a closed and an open subset of $X$, hence locally compact. One checks without difficulties that the $A_i$ are pairwise disjoint and that $X=\bigcup_{i=1}^r A_i$. Moreover, $A_i \subset U_i$ for each $i$, so $s_i|_{A_i}:A_i \to E$ is a continuous local section of $p$ for each $i \in \{1,2,\dots,r\}$.
\end{proof}

The following proposition establishes a lower bound on $\GC(M,g)$ in terms of a principal $G$-bundle over $M$ that is a Riemannian submersion. This submersion property will be used in its proof to ensure the existence of horizontal lifts of curves. For each orientable $M$, its orthonormal frame bundle is an example for such a bundle with $G=SO(\dim M)$, see e.g. \cite[Example I.5.7]{KobaNomizu}.

\begin{prop}
\label{PropLowerPrincipal}
Let $(M,g)$ be a complete Riemannian manifold and let $\pi: E\to M$ be a smooth principal $G$-bundle where $G$ is a connected Lie group. Assume that $E$ is equipped with a Riemannian metric for which $\pi$ is a Riemannian submersion. Then
$$ \GC(M,g) \geq \frac{\TC(E)}{\cat(G)}. $$
\end{prop}
\begin{proof}
Let $\GC(M)=k$ and choose pairwise disjoint and locally compact subsets $A_1,\ldots,A_k\subset M\times M$ with $\bigcup_{i=1}^k A_i = M \times M$, such that for each $i \in \{1,2,\dots,k\}$ there exists a continuous geodesic motion planner $ s_i \colon A_i \to GM$. Let $v:GM\to TM$ be the velocity map and put 
$$v_i: A_i \to TM, \quad v_i := v \circ s_i \quad \forall i \in \{1,2,\dots,k\}.$$
The $v_i$'s	 are continuous by Proposition \ref{PropVelocityCont}. For each $i$ we put 
$$ B_i:= (\pi\times \id)^{-1}(A_i) = \{ (u,q)\in E\times M\ |\  (\pi(u),q)\in A_i\}.$$
Clearly the $B_i$ are again pairwise disjoint with $\bigcup_{i=1}^k B_i = E \times M$. Let $\Hor(E) \subset TE$ denote the horizontal subbundle with respect to $\pi$. Since $d\pi|_{\Hor(E)}:\Hor(E)\to TM$ maps $\Hor_u(E)$ isomorphically onto $T_{\pi(u)}M$ for each $u \in E$, we obtain continuous lifts of the $v_i$ by
$$w_i: B_i \to \Hor(E), \quad w_i(u,q) = (d\pi|_{\Hor_u(E)})^{-1}v_i(\pi(u),q) \quad \forall i \in \{1,2,\dots,k\}. $$ 
For each $u \in E$ we let $\exp_u:T_uE\to E$ be the exponential map of the given Riemannian metric on $E$. With $PE=C^0([0,1],E)$ we define continuous maps by
$$ \eta_i: B_i \to PE, \quad (\eta_i (u,q))(t) = \exp_u(t\, w_i(u,q)) \quad  \forall (u,q)\in B_i, \ t \in [0,1], \ i \in \{1,2,\dots,k\}.$$
Each $\eta_i$ induces a continuous map 
$$\alpha_i:B_i \to E \times E, \qquad \alpha_i(u,q)= ((\eta_i(u,q))(0),(\eta_i(u,q))(1)) = (u,\exp_u(w_i(u,q))). $$ 
Since horizontal geodesics in $E$ project to geodesics in $M$, we compute that
$$(\id\times\pi)(\alpha_i(u,q)) = (\id\times\pi)(u,\eta_i(u,q)(1)) = (u,(s_i(\pi(u),q)(1)) = (u,q) \quad \forall (u,q) \in B_i.$$
Here we used that $ \pi(\eta_i(u,q)) = s_i(u,q)$ for all $(u,q) \in B_i$.
Hence, for each $i \in \{1,2,\dots,k\}$ the map $\alpha_i$ is a continuous local section of $\id \times \pi:E\times E\to E\times M$, which is again a principal $G$-bundle. The right $G$-action on $E \times E$ is given by $E \times E \times G \to E \times E$, $(u,v,h) \mapsto (u,vh)$, where we consider the right $G$-action on $E$ given by the bundle structure. Thus, we get a local trivialization of $\id \times \pi$ over each $B_i$, given explicitly by the homeomorphism
$$\Phi_i: B_i\times G \to  E\times E|_{B_i}, \qquad \Phi_i(u,q,h) =  \alpha_i(u,q)h = (u,\exp_u(w_i(u,q) )h).$$
Put $\ell=\cat(G)$. Let $e \in G$ be the unit, $P_eG=\{\gamma \in PG \ | \ \gamma(0)=e\}$ and $$q:P_eG \to G, \qquad q(\gamma)=\gamma(1).$$	 Since $P_eG$ is contractible, it holds by \cite[Theorem 18]{SchwarzGenus} that $\cat(G) = \secat(q:P_eG \to G)$. By Lemma \ref{LemmaSecatDecomp}, there are pairwise disjoint and locally compact subsets $C_1,\ldots,C_\ell \subset G$ with $\bigcup_{j=1}^\ell C_j=G$, such that for each $j \in \{1,2,\dots,\ell\}$ there is a continuous local section $r_j: C_j\to P_eG$ of $q$.	

If we put $D_{i,j}:=\Phi_i(B_i\times C_j)\subset E \times E$ for all $i \in \{1,2,\dots,k\}$ and $j \in \{1,2,\dots,\ell\}$, then the $D_{i,j}$ are pairwise disjoint, locally compact and satisfy $\bigcup_{i=1}^k \bigcup_{j=1}^\ell D_{i,j}=E \times E$. For all $i$ and $j$ we further consider the map 
$$\sigma_{i,j}: B_i \times C_j \to PE, \qquad (\sigma_{i,j}(u,q,h))(t)= (\eta_i(u,q))(t) \cdot (r_j(h))(t) \quad \forall (u,q) \in B_i, \ h \in C_j. $$
Then 
\begin{align*}
(\sigma_{i,j}(u,q,h)(0)&=(\eta_i(u,q))(0)=u, \\ (\sigma_{i,j}(u,g,h))(1)&= (\eta_i(u,q))(1)(r_j(h))(1)= \exp_u(w_i(u,q))h
\end{align*}
and thus $$(\sigma_{i,j}(u,q,h)(0),\sigma_{i,j}(u,q,h)(1))=\Phi_i(u,q,h) \qquad \forall (u,q) \in B_i, \ h \in C_j.$$ This shows that $\sigma_{i,j} \circ \Phi_i^{-1}|_{D_{i,j}}: D_{i,j} \to PE$ is a continuous motion planner for all $i \in \{1,2,\dots,k\}$ and $j \in \{1,2,\dots,\ell\}$. As a smooth manifold, $E$ is a Euclidean Neighborhood Retract (ENR). Since the $D_{i,j}$ are locally compact subsets of an ENR, they are ENRs themselves. Hence, it follows from \cite[Theorem 6.1]{FarberInstab} that
$$\TC(E) \leq k \cdot \ell = \GC(M) \cdot \cat(G),$$
which proves the claimed inequality.
\end{proof}

\begin{remark}
Since $\GC(M) \geq \TC(M)$ for all complete Riemannian manifolds $M$, the lower bound from Proposition \ref{PropLowerPrincipal} improves this basic inequality if and only if 
$$\frac{\TC(E)}{\cat (G)} > \TC(M) \qquad \Leftrightarrow \qquad \TC(E) > \cat (G)  \TC(M)= \TC(G) \TC(M),$$
where we used \cite[Lemma 8.2]{FarberInstab}. Note that the assumption on the bundle to be \emph{principal} in the previous result is necessary as the following example shows. Consider the Klein bottle $K$, which is given as a fiber bundle over $S^1$ with fiber $S^1$ and  satisfies $\TC(K)=5$ by \cite{CohenVandem}, while $\TC(S^1)=2$.  Since the round metric $g_r$ on $S^1$ satisfies $$\GC(S^1,g_r)=\TC(S^1)=2 < \frac{5}{2}=\frac{\TC(K)}{\cat( S^1)}$$ by \cite[Proposition 4.1]{RecioMitter}, the inequality from Proposition \ref{PropLowerPrincipal} would indeed be false in this situation. However, $K$ is \emph{not} given as a \emph{principal} $S^1$-bundle over $S^1$, so Proposition \ref{PropLowerPrincipal} is not applicable to this setting. By the classification theorem for principal bundles, see \cite[Theorem 14.4.1]{tDAT} the set of isomorphism classes of principal $S^1$-bundles over $S^1$ is in bijection with the set of homotopy classes $[S^1,BS^1]=[S^1,\CP^\infty]$. But $\CP^\infty$ is simply connected, so it follows that $[S^1,\CP^\infty]$ has only one element. Thus, every principal $S^1$-bundle over $S^1$ is trivial. Since $\pi_1(K) \not\cong \ZZ^2= \pi_1(S^1 \times S^1)$, the bundle $K$ is a non-trivial $S^1$-bundle. Hence, it can not be principal.
\end{remark}

Our next aim is to derive a lower bound on geodesic complexity from the structure of the cut locus of a point in the manifold. We first introduce the notion of stratification that we are using in this article.

\begin{definition}
\label{DefStrat}
Let $M$ be a manifold and let $B\subset M$ be a subset. A \emph{stratification of $B$ of depth $N \in \NN$} is a family $(S_1,\dots,S_N)$ of locally closed and pairwise disjoint subsets of $M$, such that the following conditions hold: 
\begin{enumerate}[(i)]
\item $\displaystyle B= \bigcup_{i=1}^N S_i$ \quad and \quad $\displaystyle\overline{S}_i = \bigcup_{j=i}^{N} S_j$ \quad $\forall i\in \{1,2,\dots,N\}$.
\item Let $i,j \in \{1,2,\dots,N\}$. If $Z_j$ is a connected component of $S_j$ and $Z_i$ is a connected component of $S_i$ with $Z_j\cap \overline{Z}_i \neq \emptyset$, then $Z_j \subset \overline{Z}_i$.
\end{enumerate}
\end{definition}
\begin{example}
	Let $M=\RR^2$ and let $B= [-1,1]^2$. Consider 
	\begin{align*}
	&S_1= (-1,1)\times (-1,1), \qquad S_2 = ((-1,1) \times \{-1,1\} ) \cup (\{-1,1\}\times (-1,1)), \\ &S_3 = \{(-1,-1),(-1,1),(1,-1),(1,1)\}.
		\end{align*}
One checks without difficulties that $(S_1,S_2,S_3)$ has properties (i) and (ii) from Definition \ref{DefStrat}. Hence, $(S_1,S_2,S_3)$ is a stratification of $B$.
\end{example}

Given a stratification of the cut locus of a point, we want to introduce an additional condition on those parts of the corresponding tangent cut locus that are mapped to the same stratum. This will be the crucial step for finding a lower bound for geodesic complexity. The following notion is an analogue of \cite[Definition 3.10]{RecioMitter}, see the introduction of this article and Remark \ref{RemarkIncons}.(2) below for a comparison of the two notions. The terms from Riemannian geometry that are used are to be found for example in \cite[p. 310]{LeeRiemann}.
\begin{definition}
\label{DefInconsistent}
Let $(M,g)$ be a complete Riemannian manifold, $p \in M$ and let $\Sc=(S_1,\dots,S_N)$ be a stratification of $\Cut_p(M)$. Let $K\subset T_pM$ denote the union of the tangent cut locus $\Cuttil_p(M)$ with the domain of injectivity of $\exp_p$ and let 
\begin{equation}
\label{EqexpK}
\exp_K:= \exp_p|_{K}:K \to M
\end{equation}
 denote the restriction.  We call $\Sc$ \emph{inconsistent} if for all $i \in \{2,3,\dots,N\}$ and $x \in S_i$ there exists an open neighborhood $U\subset M$ of $x$ with the following property: 

Let $Z_1,\dots,Z_s$ be the connected components of $U\cap S_{i-1}$. Then $x \in \overline{Z}_j$ for all $j \in \{1,2,\dots,s\}$ and
$$\Cuttil_p(M) \cap \exp_p^{-1}(\{x\}) \cap \bigcap_{j=1}^s \overline{\exp_K^{-1}(Z_j)}=\emptyset.$$
\end{definition}

In Section \ref{sec7-1}, we will encounter explicit examples of inconsistent stratifications when we consider flat tori.
Examples for cut loci with non-trivial stratifications which are not inconsistent are Berger spheres, as we shall see in Section \ref{sec7-2}.

\begin{remark} 
\label{RemarkIncons}
Let $(M,g)$ be a complete Riemannian manifold.
\begin{enumerate}
    \item If $M$ is a closed manifold, then the set $K$ from Definition \ref{DefInconsistent} will be homeomorphic to a closed ball, see \cite[Corollary 10.35]{LeeRiemann}, and the map $\exp_K$ from \eqref{EqexpK} is a surjection.
    As an example, consider the case of the round $n$-dimensional sphere $S^n$ of radius $1$.
    If $p\in S^n$ is a point, then the domain of injectivity of $\exp_p$ is an open ball of radius $\pi$ in the tangent space $T_p S^n$.
    The tangent cut locus $\Cuttil_p(S^n)$ is the $(n-1)$-sphere of radius $\pi$ in $T_p M$.
    Consequently, the set $K$ in this example is the closed ball of radius $\pi$ in $T_p M$.
    \item In \cite[Definition 3.8]{RecioMitter} the author introduced the concept of a level-wise stratified covering for arbitrary surjective maps. He then applied this concept to the restriction of the path fibration
$$    \pi: GX\to X\times X, $$
where $X$ is a geodesic space and $GX$ is the space of geodesic paths in $X$.

To work with this notion, one must study a stratification of the total cut locus of $X$ and explore covering properties of the restrictions of $\pi$ to its preimage.
In contrast, the above Definition \ref{DefInconsistent} for Riemannian manifolds only requires a stratification of the cut locus of a single point $p$ in a Riemannian manifold $M$ as well as properties of the Riemannian exponential map $\exp_p$.
Thus, for complete Riemannian manifolds the above definition seems easier to verify than the corresponding notion from \cite{RecioMitter}.
\end{enumerate}
\end{remark}

The following result is an analogue of the corresponding result of Recio-Mitter, see \cite[Corollary 3.14]{RecioMitter}. The proof requires $M$ to be compact, since we will use the property mentioned in Remark \ref{RemarkIncons}.(1). We recall the notation that $\GC_p(A) = \GC_M(\{p\} \times A)$ for all $A \subset M$.

\begin{theorem}
\label{TheoremLowerBound}
Let $(M,g)$ be a closed Riemannian manifold. Assume that there exists $p \in M$ for which $\Cut_p(M)$ admits an inconsistent stratification of depth $N \in \NN$. Then
$$\GC(M)\geq \GC_p(M) \geq  N+1.$$
\end{theorem}

\begin{proof}
Let $(S_1,\dots,S_N)$ be an inconsistent stratification of $\Cut_p(M)$. Assume that there are pairwise disjoint locally compact sets $E_1,E_2,\dots,E_r \subset M$ with $ \bigcup_{i=1}^r E_i=M$, such that for each $i \in \{1,2,\dots,r\}$ there exists a continuous geodesic motion planner $s_i: \{p\}\times E_i \to GM$. 

We want to show by induction that for all $k \in \{1,2,\dots,N\}$ and all $x \in S_k$ it holds that
\begin{equation}
\label{EqInd}
\#\{i \in \{1,2,\dots,r\} \ | \ x \in \overline{E}_i \}  \geq k+1.
\end{equation}
Consider the base case of $k=1$ and assume by contradiction that there was an $i \in \{1,2,\dots,r\}$ with $x \in \overline{E}_i$, but $x\notin \overline{E}_j$ for all $j \neq i$. Then $x$ has an open neighborhood $U \subset M$, such that $U \subset E_i$ and the restriction $s_i|_{\{p\} \times U}$ is a continuous geodesic motion planner on $\{p\} \times U$. But since $x \in \Cut_p(M)$, this contradicts Proposition \ref{PropContMPCut}. Hence, $\# \{i \in \{1,2,\dots,r\} \ | \ x \in \overline{E}_i\} \geq 2$, which we wanted to show.

Assume as induction hypothesis that for some $k\in \{2,3,\dots,N\}$  we have shown that 
$$\# \{ i \in \{1,2,\dots,r\} \ |\ y \in \overline{E}_i\} \geq k \quad \forall y \in S_{k-1}.$$
Let $x \in S_k$. Assume that \eqref{EqInd} was false and assume up to reordering that $x \notin \overline{E}_i$ for all $i >k$. 
Then there exists an open neighborhood $U$ of $x$ with $U\subset \bigcup_{i=1}^k E_i$. By the induction hypothesis, this yields 
\begin{equation}
\label{EqIndHyp}
 U \cap S_{k-1} \subset \overline{E}_i \quad \forall i \in \{1,2,\dots,k\}.
 \end{equation}
We assume w.l.o.g. that $U$ is chosen as in Definition \ref{DefInconsistent}, since this can be achieved by shrinking $U$. We further assume that $x \in E_1$. Let $Z_1,\dots,Z_{s}$ be the connected components of $U\cap S_{k-1}$, where $s\in \NN$ is suitably chosen. 

Let $j \in \{1,2,\dots,s\}$ and let $(a_n)_{n \in \NN}$ be a sequence in $Z_{j}$ with $\lim_{n \to \infty} a_n=x$, which exists by our choice of $U$. For all $n \in \NN$ it further holds by \eqref{EqIndHyp} that $a_n \in \overline{E}_1$. Thus, for each $n$ there exists a sequence $(b^n_m)_{m \in \NN}$ in $U \cap E_1$ with $\lim_{m \to \infty} b^n_m=a_n$. 
Put 
$$v_1 : E_1 \to T_pM, \quad v_1(y):= (v\circ s_1)(p,y),$$ 
where $v$ is the velocity map. By Proposition \ref{PropVelocityCont}, $v_1$ is continuous. Let $\exp_K:K \to M$ be given as in \eqref{EqexpK}. The set $K$ is homeomorphic to a closed ball in $T_pM$, see Remark \ref{RemarkIncons}.(1). By construction, $v_1(y) \in K$ for each $y \in E_1$, hence $(v_1(b^n_m))_{m \in \NN}$ is a sequence in $K$ for each $n \in \NN$. Since $K$ is compact, it has a convergent subsequence $(v_1(b^n_{m_k}))_{k \in \NN}$ for each $n \in\NN$. Put $\xi_n := \lim_{k \to \infty} v_1(b^n_{m_k})$ for all $n \in \NN$.  By continuity of the exponential map,
$$\exp_K(\xi_n)=\exp_p(\xi_n) = \lim_{k \to \infty}\exp_p( v_1(b^n_{m_k})) = \lim_{k \to \infty} b^n_{m_k}= a_n \qquad \forall n \in \NN.$$
Thus, 
$$\xi_n \in K \cap \exp_K^{-1}(\{a_n\}) \subset K \cap \exp_K^{-1}(Z_j) \quad \forall n \in \NN.$$
$(\xi_n)_{n \in \NN}$ is a sequence in $K$, so it has a convergent subsequence $(\xi_{n_\ell})_{\ell \in \NN}$. With $\xi_0 := \lim_{\ell \to \infty} \xi_{n_\ell}$ we obtain 
$$\exp_p(\xi_0)=\lim_{\ell \to \infty} \exp_p(\xi_{n_\ell})= \lim_{\ell \to \infty} a_{n_\ell}=x.$$
In particular, it follows from $x \in \Cut_p(M)$ that $\xi_0 \in \Cuttil_p(M)$. Since $\xi_n \in \exp_K^{-1}(\{a_n\})$ for each $n \in \NN$, we conclude that
$$\xi_0 \in \Cuttil_p(M) \cap\exp_p^{-1}(\{x\}) \cap \overline{\exp_K^{-1}(Z_j)}.$$
Note that $\xi_0$ depends on the choice of $j$. To conclude, we still need to show that the same $\xi_0$ can be chosen for each $j \in \{1,2,\dots,s\}$. We will do so by showing next that $\xi_0=v_1(x)$, which does not depend on $j$. 

Let $\dd_M: M \times M \to \RR$ be the distance function induced by the Riemannian metric.  By definition of the $\xi_{n_\ell}$, for each $\ell \in \NN$ there exists $k_\ell \in \NN$, such that $$\mathrm{d}_M(a_{n_\ell},b^{n_\ell}_{m_k}  )< \frac{1}{\ell} \quad \text{and} \quad \|\xi_{n_\ell}-v_1(b^{n_\ell}_{m_k})\| <  \frac{1}{\ell} \quad \forall k \geq k_\ell.$$
We can further choose the $k_\ell$ in such a way that $\lim_{\ell \to \infty} k_{\ell}=\infty$. By a diagonal argument, $\lim_{\ell \to \infty} b^{n_\ell}_{m_{k_\ell}} = x$. This particularly shows by continuity of $v_1$ that
$$\xi_0 = \lim_{\ell \to \infty} \xi_{n_\ell} = \lim_{\ell \to \infty} v_1(b^{n_\ell}_{m_{k_\ell}})= v_1(x).$$
Thus, $v_1(x) \in \exp_p^{-1}(\{x\}) \cap \overline{\exp_K^{-1}(Z_j)}$. Since $j$ was chosen arbitrarily, it follows that
$$v_1(x) \in \Cuttil_p(M) \cap \exp^{-1}_p(\{x\}) \cap \bigcap_{j=1}^{s} \overline{\exp_K^{-1}(Z_j)}.$$
This contradicts the inconsistency of the stratification $(S_1,\dots,S_N)$. Hence, there is no such $U$, which concludes the proof of the induction step. For $k=N$, it particularly follows from \eqref{EqInd} that $r \geq N+1$. Thus, $\GC_p(M)\geq N+1$.
\end{proof}

We will see in Section \ref{sec7-1} that flat tori are indeed examples for Riemannian manifolds whose cut loci admit inconsistent stratifications. Next we will discuss a more tangible criterion on a cut locus that implies the existence of an inconsistent stratification. For this purpose, we will use results and constructions of J.-I. Itoh and T. Sakai from \cite{ItohSakai}. Large parts of these methods are extensions of those applied by V. Ozols in \cite{Ozols}.

\begin{definition}[{\cite[p. 68 and Definition 2.1]{ItohSakai}}]
Let $(M,g)$ be a complete Riemannian manifold and let $p \in M$. 
\begin{enumerate}[a)]
\item We say that $q \in \Cut_p(M)$ is \emph{of order $k+1$}, where $k \in \NN$, if there are precisely $k+1$ minimal geodesics $\gamma_0,\gamma_1,\dots,\gamma_k \in GM$ with $\gamma_i \neq \gamma_j$ if $i \neq j$ and with $\gamma_i(0)=p$ and $\gamma_i(1)=q$ for all $i \in \{0,1,2,\dots,k\}$.
\item We call $q$ \emph{non-degenerate} if the vectors $\dot\gamma_0(1),\dot\gamma_1(1),\dots,\dot\gamma_k(1)\in T_qM$ are in general position, i.e. if $\{\dot\gamma_i(1)-\dot\gamma_0(1) \ | \ i \in \{1,2,\dots,k\}\}$ is linearly independent.
\end{enumerate}
\end{definition}

As carried out by Itoh and Sakai in \cite[Remark 2.2]{ItohSakai}, a large class of  two-dimensional flat tori provides an example for manifolds with non-degenerate cut points. However, our study of flat tori in Section \ref{sec7-1} will not rely on this notion of non-degeneracy, but will employ the above inconsistency condition directly. \medskip

We recall that a \emph{conjugate point} of a point $p$ in a Riemannian manifold $(M,g)$ is a point $q \in M$, such that there is a geodesic segment from $p$ to $q$ along which there exists a non-trivial Jacobi field which vanishes in $p$ and $q$, see \cite[p. 298]{LeeRiemann}.

\begin{remark}
\begin{enumerate} 
\item As shown by A. Weinstein in \cite[p. 29]{WeinsteinCut}, every closed manifold $M$ with $\dim M \geq 2$ and not homeomorphic to $S^2$ admits a Riemannian metric for which there exists $p \in M$ such that $\Cut_p(M)$ does not contain any conjugate points. Itoh and Sakai conjectured in \cite[Remark 2.9]{ItohSakai} that the set of all such metrics on $M$ contains as a dense subset the set of those metrics for which all points in $\Cut_p(M)$ are non-degenerate.
\item It is evident from the definition of non-degeneracy that the order of a non-degenerate cut point is at most $\dim M+1$.
\end{enumerate}
\end{remark}

\begin{theorem}
Let $(M,g)$ be a closed Riemannian manifold and assume that there exists $p \in M$ for which $\Cut_p(M)$ does not contain any conjugate points of $p$ and for which all points in $\Cut_p(M)$ are non-degenerate. 
Let $N:= \max \{k \in \NN \ | \ \exists q \in \Cut_p(M) \text{ of order } k+1\}$. Then $\Cut_p(M)$ admits an inconsistent stratification of depth $N$. 
\end{theorem}

\begin{proof}
Let $\mathcal{C}:=(C_1,\dots,C_N)$ be given by 
$$C_k:= \{q \in \Cut_p(M)\ | \ q \text{ is of order } k+1\} \quad \forall k \in \{1,2,\dots,N\}.$$
It is shown in \cite[Proposition 2.4]{ItohSakai} that under the non-degeneracy assumption on the points in $\Cut_p(M)$, $\mathcal{C}$ is a Whitney stratification of $\Cut_p(M)$, as defined in \cite[p. 37]{GoreskyMP}. Hence, $\mathcal{C}$ is in particular an $\mathscr{S}$-decomposition in the sense of Goresky and MacPherson, see \cite[p. 36]{GoreskyMP}. One checks immediately that the two conditions defining such an $\mathscr{S}$-decomposition imply that $\mathcal{C}$ is a stratification of $\Cut_p(M)$ in the sense of Definition \ref{DefStrat}.
It remains to show that $\mathcal{C}$ is inconsistent.  Fix $k\in \{1,2,\dots,N\}$, let $q \in C_k$ and let $\gamma_0,\gamma_1,\dots,\gamma_k:[0,1] \to M$ be geodesics from $p$ to $q$ with $\gamma_i \neq \gamma_j$ whenever $i \neq j$. For each $i \in \{0,1,\dots,k\}$ put $v_i := \dot\gamma_i(0) \in T_pM$, such that 
$$\Cuttil_p(M) \cap \exp_p^{-1}(\{q\})=\{v_0,v_1,\dots,v_k\}.$$
Choose an open neighborhood $U$ of $q$, such that $U \cap C_k$ is connected and such that 
\begin{equation}
\label{EqCutplocal}
\Cut_p(M) \cap U = \bigcup_{i=1}^k C_i \cap U.
\end{equation}
Such a neighborhood exists by the stratification properties. As discussed in \cite[p. 68]{ItohSakai}, since $q$ is non-degenerate, we can choose an open neighborhood $V_i\subset T_pM$ of $v_i$ for each $i \in \{0,1,\dots,k\}$, such that $\exp_p$ maps $V_i$ diffeomorphically onto $U$. Put $F_i:= (\exp|_{V_i})^{-1}:U \to V_i$. As explained in \cite[p.220f.]{Ozols}, up to shrinking $U$ we can assume that every minimal geodesic $\gamma$ from $p$ to an element of $U$ has $\dot\gamma(0) \in \bigcup_{i=0}^k V_i$. We further assume that $\overline{V}_i \cap \overline{V}_j= \emptyset$ whenever $i \neq j$.  For $i \in \{1,2,\dots,k\}$ we define $$f_i: U \to \RR,  \quad f_i(x) = \|F_i(x)\|- \|F_0(x) \|,$$ where $\|\cdot\|$ denotes the norm on $T_pM$ defined by the Riemannian metric. 
 With $f:U \to \RR^k$, $f=(f_1,f_2,\dots,f_k)$, it follows that $f^{-1}(\{0\}) = C_{k}\cap U.$ 	
For $i \in \{1,2,\dots,k\}$ we further let 
$$g_i: U \to \RR^{k-1}, \qquad g_i = (f_1,\dots,f_{i-1},f_{i+1},\dots,f_k)$$
and put 
$$g_0: U \to \RR^{k-1}, \quad g_0(x) = \left( \|F_2(x)\|-\|F_1(x)\|,\|F_3(x)\|-\|F_1(x)\|,\dots,\|F_k(x)\|-\|F_1(x)\|\right).$$
Then, by assumption on $U$,  
$$C_{k-1}\cap U = \bigcup_{i=0}^k g_i^{-1}(\{0\}) \setminus C_{k}= \bigcup_{i=0}^k g_i^{-1}(\{0\}) \setminus f^{-1}(\{0\}).$$
The connected components of $C_{k-1}\cap U$ are the sets $Z_0,Z_1,\dots,Z_k$, where
\begin{align*}
Z_i := g_i^{-1}(\{0\}) \cap f_i^{-1}(0,+\infty) \ \ \forall i \in \{1,2,\dots,k\}, \quad Z_0 := g_0^{-1}(\{0\})\cap f_1^{-1}(-\infty,0).
\end{align*}
By construction of the sets, 
$$\Cuttil_p(M) \cap \exp_p^{-1}(Z_i) \subset \bigcup_{j \neq i} V_j \quad \forall i \in \{0,1,\dots,k\}.$$ 
A closer investigation, using that $\Cuttil_p(M) \cap \exp_p^{-1}(\{q\})=\{v_0,v_1,\dots,v_k\}$ and that the closures of the $V_i$ are pairwise disjoint, shows that
$$\Cuttil_p(M) \cap \exp^{-1}_p(\{q\}) \cap \overline{\exp_p^{-1}(Z_i)}=\{v_0,v_1,\dots,v_{i-1},v_{i+1},\dots,v_k\} \qquad \forall i \in \{0,1,\dots,k\}.$$
This implies $$\Cuttil_p(M) \cap \exp^{-1}_p(\{q\}) \cap \bigcap_{i=0}^k \overline{\exp_p^{-1}(Z_i)}=\emptyset.$$ 
Since $k$ and $q$ were chosen arbitrarily, this shows the inconsistency of $\mathcal{C}$.
\end{proof}

Combining the previous theorem with our lower bound from Theorem \ref{TheoremLowerBound} yields:

\begin{cor}
\label{CorLowerOrder}
Let $(M,g)$ be a closed Riemannian manifold and assume that there exists $p \in M$, such that $\Cut_p(M)$ does not contain any conjugate points of $p$ and such that all points in $\Cut_p(M)$ are non-degenerate. If $\Cut_p(M)$ contains a point of order $k+1$, where $k \in \NN$, then 
$$\GC(M,g) \geq k+1.$$
\end{cor}

\section{An upper bound for homogeneous Riemannian manifolds}
\label{sec5}
From this section on, we will mostly consider homogeneous Riemannian manifolds and exploit their symmetry properties. Given a Riemannian manifold $(M,g)$, we let $\Isom(M):= \Isom(M,g)$ denote its group of isometries and consider it as a subspace of $C^0(M,M)$ with the compact-open topology. We recall that $(M,g)$ is called \emph{homogeneous} if $\Isom(M)$ acts transitively on $M$. Note that every homogeneous Riemannian manifold is necessarily complete, see \cite[Theorem IV.4.5]{KobaNomizu}.

Having derived lower bounds for geodesic complexity in the previous section, we next want to find upper bounds. After some preparatory lemmas, we will establish an upper bound on $\GC(M)$ for a homogeneous Riemannian manifold $M$ in terms of the subspace complexity $\GC_M(\{p\} \times \Cut_p(M))$ and a categorical invariant determined by its isometry action. Intuitively, the transitivity of the isometry action implies that continuous geodesic motion planners on subsets of cut loci of single points can be continuously extended to larger subsets of the total cut locus. We will then go on to study further upper bounds on $\GC(M)$ in the case that $\Cut_p(M)$ admits a stratification. The following is a folklore result from Riemannian geometry. 
\begin{lemma}
Let $(M,g)$ be a homogeneous Riemannian manifold and let $p \in M$. Then 
$$\ev_p: \Isom(M) \to M, \qquad \ev_p(\phi)=\phi(p),$$
is a principal $\Isom_p(M)$-bundle, where $\Isom_p(M)$ denotes the isotropy group of the isometry action on $M$ in $p$.
\end{lemma}
\begin{proof}
By \cite[Theorem 21.17]{LeeSmooth}, $\ev_p$ induces an $\Isom(M)$-equivariant diffeomorphism $f: \Isom(M)/\Isom_p(M) \to M$. Moreover, the projection $q:\Isom(M) \to \Isom(M)/\Isom_p(M)$ is a principal $\Isom_p(M)$-bundle by \cite[Example I.5.1]{KobaNomizu}. One easily shows that $\ev_p = f \circ q$, which implies that $\ev_p$ is a principal $\Isom_p(M)$-bundle as well. 
\end{proof}

\begin{example}
\label{ExampleLeftInv}
Given a Lie group $G$ with a left-invariant Riemannian metric,  the left-multiplication $\ell_g:G \to G$, $\ell_g(h)=gh$, is an isometry for each $g \in G$. With $e \in G$ denoting the unit, one further derives from $\ell_g(e)=g$ for each $g \in G$ that the map $s: G \to \Isom(G)$, $s(g)= \ell_g$, is a continuous section of the bundle $\ev_e:\Isom(G)\to G$.
\end{example}

\begin{lemma}
\label{LemmaShiftGeodesics}
Let $A,B \subset M$ and $p \in M$. Assume that there are a continuous geodesic motion planner $\sigma_B:\{p\} \times B \to GM$ and a continuous local section $s:A \to \Isom(M)$ of $\ev_p$. Then there exists a continuous geodesic motion planner $\sigma: F \to GM$, where $$F := \{(x,y) \in M \times M \ | \ x \in A, \ y \in s(x)\cdot B\}.$$
\end{lemma}
\begin{proof}
We define $\sigma:F \to GM$ by 
$$\sigma(x,y) = s(x)\circ \sigma_B(p,s(x)^{-1}\cdot y) \qquad \forall (x,y) \in F.$$
By construction $\sigma_B(p,s(x)^{-1}\cdot y)$ is a minimal geodesic from $p$ to $s(x)^{-1}\cdot y$. Since $s(x)$ is an isometry for each $x$, $\sigma(x,y)$ is indeed a minimal geodesic from 
$$s(x)\cdot p = \ev_p(s(x))= x\qquad \text{to} \qquad s(x)\cdot (s(x)^{-1}\cdot y)=y.$$ 
So $\sigma$ is a geodesic motion planner and it only remains to show its continuity. 

Let $\rho:\Isom(M) \times M \to M$ denote the action of the isometry group by evaluation and let again $PM=C^0([0,1],M)$. 
By \cite[Theorem VII.2.10]{Bredon} the composition map
$$  \varphi : C^0(M,M)\times PM \to PM,   \qquad  \varphi (f,\gamma) = f\circ \gamma, $$
is continuous with respect to the compact-open topologies. Thus, the restriction of $\varphi$ to
$$   \Isom(M)\times GM \subset C^0(M,M)\times PM   $$
defines a continuous action $$\tilde{\rho } : \Isom(M) \times GM \to GM, \qquad 
\tilde{\rho} = \varphi|_{\Isom(M)\times GM}.$$
The inversion $i:\Isom(M) \to \Isom(M)$, $i(g)=g^{-1}$, is continuous since $\Isom(M)$ is a topological group. We can express $\sigma$ as 
$$\sigma(x,y)= \tilde\rho(s(x),\sigma_B(p,\rho(i(s(x)),y)) \qquad \forall (x,y) \in F.$$
All maps on the right-hand side are continuous, so $\sigma$ is continuous as well.
\end{proof}

The previous lemma allows us to make a useful estimate between the subspace geodesic complexity of the total cut locus and the one of one single cut locus in the homogeneous case.

\begin{theorem} 
\label{TheoremGCCut}
Let $(M,g)$ be a homogeneous Riemannian manifold and let $p \in M$. Then
$$\GC(M) \leq \secat(\ev_p:\Isom(M) \to M) \cdot \GC_p(\Cut_{p}(M))+1.$$
\end{theorem}
\begin{proof}
As seen in Remark \ref{RemarkOCut}, it holds that $\GC(M) \leq \GC_M(\Cut(M))+1$, so it suffices to show that 
$$\GC_M(\Cut(M)) \leq \secat(\ev_p) \cdot \GC_p(\Cut_{p}(M)).$$
Let $k := \secat(\ev_p)$ and $r:= \GC_p( \Cut_p(M))$. By Lemma \ref{LemmaSecatDecomp}, there are pairwise disjoint locally compact $A_1,\dots,A_k \subset M$ with $M = \bigcup_{i=1}^k A_i$, for which there is a continuous local section $s_i:A_i \to \Isom(M)$ of $\ev_p$ for each $i \in \{1,2,\dots,k\}$.
 Let $B_1,\dots,B_r\subset M$ be pairwise disjoint and locally compact with $\Cut_p(M) \subset \bigcup_{j=1}^r B_j$, such that for each $j$ there exists a continuous geodesic motion planner  $\sigma_j:\{p\}\times B_j \to GM$. Put
$$F_{i,j} := \left\{(x,y) \in M\times M \ | \ x \in A_i, \ y \in s_i(x)\cdot B_j \right\} \qquad \forall i \in \{1,2,\dots,k\}, \ j \in \{1,2,\dots,r\}.$$
By construction, the elements of $\{F_{i,j}\ | \ i \in \{1,2,\dots,k\}, \ j \in \{1,2,\dots,r\}\}$ are pairwise disjoint.
Furthermore, for all $i\in\{1,2,\dots,k\}$ and $j\in\{1,2,\dots,r\}$ the following map is a homeomorphism: $$\psi_{i,j}: A_i\times B_j \to F_{i,j}, \qquad   \psi_{i,j}(x,y)  = (x,s_i(x)\cdot y).     $$
Consequently, the $F_{i,j}$ are locally compact. If $(x,y)\in\Cut(M)$, then $x\in A_i$ for some $i\in\{1,2,\dots,k\}$.
Since $s_i(x)^{-1}$ is an isometry, it holds that $s_i(x)^{-1}\cdot y\in \Cut_p(M) $.
Hence, there is a $j\in\{1,2,\dots,r\}$ with $s_i(x)^{-1}\cdot y\in B_j$ and therefore $(x,y)\in F_{i,j}$ by definition.
This shows that
$$\Cut(M) \subset \bigcup_{i=1}^k \bigcup_{j=1}^r F_{i,j}.$$
Moreover, by Lemma \ref{LemmaShiftGeodesics} we can find a continuous geodesic motion planner $F_{i,j} \to GM$ of $p$ for all $i$ and $j$. Thus, $\GC_M(\Cut(M)) \leq k r$, which shows the claim.
\end{proof}

The previous upper bound has a  particularly strong consequence for connected Lie groups.

\begin{cor}
\label{CorLieGroup}
Let $G$ be a connected Lie group equipped with a left-invariant Riemannian metric and let $e\in G$ denote the unit element. Then
$$\GC(G) \leq \GC_e(\Cut_e(G))+1.$$
\end{cor}
\begin{proof}
This is an immediate consequence of Theorem \ref{TheoremGCCut}. Since $\ev_e: \Isom(G) \to G$ admits a continuous section, see Example \ref{ExampleLeftInv}, it follows that $\secat(\ev_e)=1$. 
\end{proof}

Sectional categories of fibrations are in general hard to compute. A common way of estimating their values from above is by the Lusternik-Schnirelmann categories of their base spaces using \cite[Theorem 18]{SchwarzGenus}. In our situation, this leads to the following estimate.

\begin{cor}
Let $(M,g)$ be a homogeneous Riemannian manifold and let $p \in M$. Then 
$$\GC(M) \leq \cat(M) \cdot \GC_p(\Cut_p(M))+1.$$
\end{cor}
\begin{proof}
This is an immediate consequence of Theorem \ref{TheoremGCCut} and the fact that every fibration $p:E \to B$ satisfies $\secat(p) \leq \cat(B)$ by \cite[Theorem 18]{SchwarzGenus}.
\end{proof}

We want to further estimate geodesic complexity from above by finding upper bounds for subspace geodesic complexities of cut loci. In case $\Cut_p(M)$ admits a stratification, we can compare $\GC_p(\Cut_p(M))$ to the subspace geodesic complexities of its strata.

\begin{prop}
\label{PropCutStrat}
Let $(M,g)$ be a complete Riemannian manifold, let $p \in M$ and assume that $\Cut_p(M)$ has a stratification $(S_1,\dots,S_k)$ of depth $k$. Then 
$$\GC_p(\Cut_p(M)) \leq \sum_{i=1}^k \max_{Z_i \in \pi_0(S_i)} \GC_p(Z_i),$$
where $\pi_0(X)$ denotes the set of connected components of a space $X$.
\end{prop}

\begin{proof}
Since $\Cut_p(M)= S_1 \cup \dots \cup S_k$, it follows from Remark \ref{RemarkGCelem}.(3) that $\GC_p(\Cut_p(M)) \leq \sum_{i=1}^k \GC_p(S_i).$ Now fix $i \in \{1,2,\dots,k\}$ and let $Z_1,\dots,Z_r$ be the connected components of $S_i$. Put $$s_i := \max_{j \in \{1,2,\dots,r\}} \GC_p(Z_j).$$ For each $j \in \{1,2,\dots,r\}$ let $A^j_1,\dots,A^j_{s_i} \subset Z_j$ be pairwise disjoint and locally compact, such that for each $j \in \{1,2,\dots,r\}$ and $\ell \in \{1,2,\dots,s_i\}$ either $A^j_\ell=\emptyset$ or there exists a continuous geodesic motion planner $\sigma_{j,\ell}:\{p\}\times A^j_\ell \to GM$. 
Put $A_\ell := \bigcup_{j=1}^r A_\ell^j$ for each $\ell \in \{1,2,\dots,s_i\}$. Then the $A_\ell$ are pairwise disjoint and locally compact with $S_i = \bigcup_{\ell=1}^{s_i} A_\ell$. Moreover, since by definition of a stratification, $Z_i \cap \overline{Z}_j=\emptyset$ for all $i \neq j$, the maps
$$\sigma_\ell:\{p\}\times A_\ell \to GM, \qquad \sigma_\ell(p,x) = \sigma_{j,\ell}(p,x) \quad \forall x \in A^j_\ell, \ j \in \{1,2,\dots,r\},$$
are well-defined continuous geodesic motion planners. This shows $\GC_p(S_i) \leq s_i$ for each $i \in \{1,2,\dots,k\}$, which implies the claim.
\end{proof}

\begin{cor}
\label{CorUpperStrat}
Let $(M,g)$ be a homogeneous Riemannian manifold, let $p \in M$ and assume that $\Cut_p(M)$ has a stratification $(S_1,\dots,S_k)$ of depth $k$. Then 
$$\GC(M) \leq \secat(\ev_p:\Isom(M)\to M)\cdot  \sum_{i=1}^k \max_{Z_i \in \pi_0(S_i)} \GC_p(Z_i)+1.$$
\end{cor}
\begin{proof}
This follows from Theorem \ref{TheoremGCCut} and Proposition \ref{PropCutStrat}.
\end{proof}

\section{Trivially covered stratifications}  
\label{sec6}

In \cite{RecioMitter}, Recio-Mitter considered total cut loci with stratifications whose strata are finitely covered by the path fibration. As a part of \cite[Corollary 3.14]{RecioMitter}, he showed that if that stratification is inconsistent and trivially covered, this knowledge about the total cut locus suffices to compute the geodesic complexity of the space. 

In this section, we will revisit the notion of trivially covered stratifications in the setting of Riemannian manifolds, but in contrast to \cite{RecioMitter}, we will put a covering condition on the cut locus of a single point instead of the total cut locus. We will then derive an upper bound for the numbers $\GC_p(M)$ that we have studied in the previous section. From this estimate we will derive an upper bound for the geodesic complexity of homogeneous Riemannian manifolds for which the cut locus of a point admits a trivially covered stratification.

\begin{definition}
Let $M$ be a complete Riemannian manifold, let $p \in M$ and let $\Sc= (S_1,\dots,S_N)$ be a stratification of $\Cut_p(M)$. We call $\Sc$ \emph{trivially covered} if for all $k \in \{1,2,\dots,N\}$ and for all connected components $Z$ of $S_k$ the restriction 
$$ \exp_p|_ {\Cuttil_p(M) \cap \exp_p^{-1}(Z)}:  \Cuttil_p(M) \cap \exp_p^{-1}(Z) \to Z$$
is a trivial covering. Here, a trivial covering is understood to be a covering $q: X \to Y$, for which there is a discrete set $D$ and a homeomorphism $f: X \to Y \times D$, such that $q= \mathrm{pr} \circ f$, where $\mathrm{pr}:Y \times D \to Y$ is the projection onto the first factor.
\end{definition}

\begin{theorem}
\label{TheoremTrivCover}
Let $M$ be a complete Riemannian manifold, let $p \in M$ and assume that $\Cut_p(M)$ admits a trivially covered stratification of depth $N \in \NN$. Then
$$\GC_p(\Cut_p(M)) \leq N.$$
\end{theorem}
\begin{proof}
Let $\Sc=(S_1,\dots,S_N)$ be a trivially covered stratification of $\Cut_p(M)$. We want to show that $\{p\} \times S_k$ admits a continuous geodesic motion planner for each $k \in \{1,2,\dots,N\}$. 
For a fixed $k\in \{1,2,\dots,N\}$ let $Z_1,\dots,Z_r$ be the connected components of $S_k$ for suitable $r \in \NN$. For $i \in \{1,2,\dots,r\}$ let $B_i$ be an arbitrary sheet of the trivial covering
$$\exp_p|_ {\Cuttil_p(M) \cap \exp_p^{-1}(Z_i)}:  \Cuttil_p(M) \cap \exp_p^{-1}(Z_i) \to Z_i.$$
Then $\exp_p|_{B_i}:B_i \to Z_i$ is a homeomorphism. With $\varphi_i := (\exp_p|_{B_i})^{-1}:Z_i \to B_i$ one checks without difficulties that 
$$s_i: \{p\} \times Z_i \to GM, \qquad (s_i(p,q))(t)= \exp_p(t\varphi_i^{-1}(q)),$$
is a continuous geodesic motion planner and thus $\GC_p(Z_i)=1$. 
Since $k \in \{1,2,\dots,N\}$ was chosen arbitrarily, the claim follows from Proposition \ref{PropCutStrat}.
\end{proof}

With the additional hypotheses that $M$ is compact and that the stratification in Theorem \ref{TheoremTrivCover} is inconsistent, one can derive an equality from Theorem \ref{TheoremTrivCover}. The following result is analogous to the corresponding part of \cite[Corollary 3.14]{RecioMitter}.

\begin{cor}
\label{CorTrivCov}
Let $M$ be a closed Riemannian manifold, let $p \in M$ and assume that $\Cut_p(M)$ admits a trivially covered inconsistent stratification of depth $N \in \NN$. Then $\GC_p(M) = N+1$.
\end{cor}
\begin{proof}
By restricting the motion planner from Remark \ref{RemarkOCut}.(2), one obtains a continuous geodesic motion planner on $\{p\} \times (M \setminus \Cut_p(M))$. It follows from Theorem \ref{TheoremTrivCover} that
$$\GC_p(M) \leq \GC_p(\Cut_p(M))+1 \leq N+1.$$
But by Theorem \ref{TheoremLowerBound}, it also holds that $\GC_p(M) \geq N+1$, which proves the equality.
\end{proof}

\begin{cor}
\label{CorLieGroupTrivCov}
Let $G$ be a compact connected Lie group equipped with a left-invariant Riemannian metric and let $e \in G$ denote the unit element. If $\Cut_e(G)$ admits a trivially covered inconsistent stratification of depth $N$, then 
$$\GC(G)=N+1.$$
\end{cor}
\begin{proof}
Combining Theorem \ref{TheoremTrivCover} with Corollary \ref{CorLieGroup} yields
$$\GC(G) \leq \GC_e(\Cut_e(G))+1 \leq N+1.$$
But by Theorem \ref{TheoremLowerBound}, $\GC(G) \geq N+1$ as well, so the claim follows.
\end{proof}

\section{Examples: flat tori and Berger spheres}	
\label{sec7}

We want to use the results of Sections \ref{sec5} and \ref{sec6} to compute the geodesic complexities of two classes of examples: two-dimensional flat tori and three-dimensional Berger spheres. The cut loci of points in such spaces are well-understood and admit stratifications of a well-behaved kind.

\subsection{Geodesic complexity of flat tori} \label{sec7-1}
Recio-Mitter has computed the geodesic complexity of a standard flat $n$-dimensional torus in \cite[Theorem 4.4]{RecioMitter}. More precisely, he has shown that the standard flat metric $g_f$ on the $n$-torus $T^n$ satisfies $\GC(T^n,g_f)=n+1$ for each $n \in \NN$.

In the course of this subsection, we will extend the two-dimensional case of Recio-Mitter's result to \emph{arbitrary} flat metrics on two-dimensional tori. The cut loci of such metrics are well-understood. 

Before we do so, we will re-obtain Recio-Mitter's computation for standard flat tori using the methods of this article. This example is particularly instructive and illustrates the use of inconsistent stratifications. Moreover, in contrast to \cite[Theorem 4.4]{RecioMitter}, we only need to consider the cut locus of a single point, while in the proof of \cite[Theorem 4.4]{RecioMitter} a stratification of $T^n\times T^n$ is required and the structure of the space of geodesic paths in $T^n$ needs to be examined.

\begin{example}
\label{ExampleFlatTorus}
Let $n \in \NN$ and consider the $n$-torus $T^n$ with the standard flat metric $g_f$, i.e. the quotient metric induced by the standard metric on $\RR^n$ and by identifying $T^n = \RR^n/(2\ZZ)^n$. Equivalently, $T^n$ is obtained from $\RR^n$ by collapsing the lattice defined by an arbitrary family of $n$ pairwise orthogonal vectors of length two). Let $\pi:\RR^n \to T^n$ be the projection and put $o:= \pi(0)$ and $M:= (T^n,g_f)$. We identify $\RR^n$ with $T_oM$ in the obvious way. 

Note that $T^n$ is isometric to the Riemannian product $(\RR/(2\ZZ))^n$.
For $N:=\RR/(2\ZZ)$ let $\mathrm{pr}:\RR\to N$ be the obvious Riemannian covering and put $p_0 := \mathrm{pr}(0)\in N$. Then $\Cut_{p_0}(N)=\{\mathrm{pr}(1)\}$ and the tangent cut locus is given by $$\Cuttil_{p_0}(N) = \{-1,1\}$$ under the identification $T_{p_0}N \cong \RR$.

Given the Riemannian product of two Riemannian manifolds $(M_1,g_1)$ and $(M_2,g_2)$ the cut locus of a point $(p_1,p_2)\in M_1\times M_2$ is easily seen to be
$$  \Cut_{(p_1,p_2)}(M_1\times M_2)  = \left(\Cut_{p_1}(M_1)\times M_2\right) \cup \left(M_1 \times \Cut_{p_2}(M_2)\right) ,  $$
see \cite[p. 328]{Crittenden}. For $i\in \{1,2\}$ let $K_i$ be the union of the injectivity domain in $T_{p_i}M_i$ with the tangent cut locus $\Cuttil_{p_i}(M_i)$. Similar to the cut locus, the tangent cut locus of $(p_1,p_2)$ is given by
 $$   \Cuttil_{(p_1,p_2)}  (M_1\times M_2)  =  \left(\Cuttil_{p_1}(M_1)\times K_2\right) \cup \left(K_1 \times \Cuttil_{p_2}(M_2)\right) $$
 under the identification $T_{(p_1,p_2)}(M_1\times M_2) \cong T_{p_1}M_1 \times T_{p_2}M_2$.
 For products of finitely many manifolds, one iteratively derives analogous results for cut loci and tangent cut loci. 

 We conclude that if $I^n:= [-1,1]^n$, then the tangent cut locus of $o$ in $M=(T^n,g_f)$ is
$$\Cuttil_o(M) = \partial I^n.$$
See also \cite[p. 107]{GHL} for the case $n=2$.
 The boundary $\partial I^n$ admits a stratification $\partial I^n = \bigcup_{k=1}^n \mathcal{A}_k$ of depth $n$, given as follows: 
For each $k \in \{1,2,\dots,n\}$ we put $$J_k := \{(i_1,\dots,i_k) \in \NN^k \ | \ 1 \leq i_1 < \dots < i_k \leq n\}.$$ Then each $\mathcal{A}_k$ is given as the disjoint union $\mathcal{A}_k = \bigcup_{(i_1,\dots,i_k) \in J_k}A_{i_1,\dots,i_k}$, where
$$A_{i_1,\dots,i_k} := \left\{(x_1,\dots,x_n) \in I^n \ \middle| \ |x_{\ell}|=1 \text{ if } \ell \in  \{i_1,i_2,\dots,i_k\}, \ |x_\ell|<1 \text{ if } \ell \notin \{i_1,\dots,i_k\}\right\}.$$

For $(i_1,\dots,i_k) \in J_k$ and $j_1,\dots,j_k \in \{-1,1\}$ we put 
$$A_{i_1,\dots,i_k,j_1,\dots,j_k}= \left\{(x_1,\dots,x_n) \in \RR^n \ \middle| \ x_{i_1}=j_1,\dots,x_{i_k}=j_k, \  |x_\ell| <1 \text{ if } \ell \notin \{i_1,\dots,i_k\} \right\}.$$
Then the sets $A_{i_1,\dots,i_k,j_1,\dots,j_k}$, where $j_1,\dots,j_k \in \{-1,1\}$, are precisely the connected components of $A_{i_1,\dots,i_k}$.


Put $\mathcal{B}_k := \exp_o(\mathcal{A}_k)$. We claim that $(\mathcal{B}_1,\dots,\mathcal{B}_n)$ is a trivially covered stratification of $\Cut_o(M)$.
One checks that the connected components of each of the $\mathcal{B}_k$ are precisely the sets
$$ B_{i_1,\dots,i_k} := \exp_o(A_{i_1,\dots,i_k})\qquad  \text{ where } (i_1,\dots,i_k) \in J_k.$$
Moreover, for all  $(i_1,\dots,i_k) \in J_k$ and all $j_1,\dots,j_k \in \{-1,1\}$ the restriction 
$$\exp_o|_{A_{i_1,\dots,i_k,j_1,\dots,j_k}}:A_{i_1,\dots,i_k,j_1,\dots,j_k} \to B_{i_1,\dots,i_k}$$
is a homeomorphism.
From the explicit description of the $A_{i_1,\dots,i_k}$ one derives that $(\mathcal{B}_1,\dots,\mathcal{B}_n)$ is a stratification. It further follows from the above observations that $$\exp_o|_{A_{i_1,\dots,i_k}}:A_{i_1,\dots,i_k}\to B_{i_1,\dots,i_k}$$ is a trivial covering map for all $(i_1,\dots,i_k) \in J_k$. Since $k \in \{1,2,\dots,n\}$ was chosen arbitrarily, this shows that $(\mathcal{B}_1,\dots,\mathcal{B}_n)$ is trivially covered.

We now want to prove that $(\mathcal{B}_1,\dots,\mathcal{B}_n)$ is indeed an inconsistent stratification of $\Cut_o(M)$. For this purpose, let $k \in \{2,3,\dots,n\}$, $(i_1,\dots,i_k) \in J_k$ and $x \in B_{i_1,\dots,i_k}$.  We assume w.l.o.g. that $(i_1,i_2,\dots,i_k)=(1,2,\dots,k)$. Then there are $y_1,\dots,y_{n-k}\in (-1,1)$, such that $$x= \exp_o(1,1,\dots,1,y_1,\dots,y_{n-k}).$$ It further holds that  
\begin{equation}
\label{EqexpKx}
 \exp_K^{-1}(\{x\}) = \left\{(j_1,\dots,j_k,y_1,\dots,y_{n-k}) \in T_oM \ \middle| \ j_1,\dots,j_k \in \{-1,1\} \right\},
 \end{equation}
where $K:= I^n$ and $\exp_K:= \exp_o|_K:K \to M$, which is a special case of the map defined in \eqref{EqexpK}. Let $i \in \{1,2,\dots,k\}$ and let $\Bhat_i := B_{1,\dots,{i-1},{i+1},\dots,k} \subset \mathcal{B}_{k-1}$. Given $\varepsilon>0$ put 
$$U_\varepsilon:= \exp_o \Big((1-\varepsilon,1+\varepsilon)^k \times \prod_{j=1}^{n-k} (y_j-\varepsilon,y_j+\varepsilon)\Big) \subset M.$$
Then $U_\varepsilon$ is an open neighborhood of $x$ and for sufficiently small $\varepsilon>0$, it holds that $\Bhat_i \cap U_\varepsilon$ has two components $C^+_{i}$ and $C^-_{i}$. With $I_+:= (1-\varepsilon,1)$ and $I_-:=(-1,-1+\varepsilon)$, we put  for all $j_1,\dots,j_{i-1},j_{i+1},\dots,j_k \in \{-1,1\}$: 
\begin{align*}
A^\pm_{j_1,\dots,j_{i-1},j_{i+1},\dots,j_k } &:= \Big\{(j_1,\dots,j_{i-1},t,j_{i+1},\dots,j_k,q) \ \Big| \ t \in I_\pm, \ q \in \prod_{\ell=1}^{n-k} (y_\ell-\varepsilon,y_\ell+\varepsilon) \Big\}. 
\end{align*}
 The two components $C^+_i$ and $C^-_i$ then satisfy 
\begin{align*}
\exp_K^{-1}(C^\pm_i)&= \bigcup_{j_1,\dots,j_{i-1},j_{i+1},\dots,j_k \in \{-1,1\}} A^\pm_{j_1,\dots,j_{i-1},j_{i+1},\dots,j_k}. 
\end{align*}
Combining this observation with \eqref{EqexpKx} yields
\begin{align*}
&\Cuttil_o(M) \cap \exp_o^{-1}(\{x\}) \cap \overline{\exp_K^{-1}(C^\pm_i)}\\ &= \left\{(j_1,\dots,j_{i-1},\pm1,j_{i+1},\dots,j_k,y_1,\dots,y_{n-k}) \ \middle| \ j_1,\dots,j_{i-1},j_{i+1},\dots,j_k \in \{-1,1\} \right\},
\end{align*}
In particular, $\exp_o^{-1}(\{x\}) \cap \overline{\exp_K^{-1}(C^+_i)} \cap \overline{\exp_K^{-1}(C^-_i)}=\emptyset$, implying that $(\mathcal{B}_1,\dots,\mathcal{B}_n)$ satisfies the inconsistency condition at $x$. Since $x\in \Cut_o(M)$ was chosen arbitrarily, this shows that $(\mathcal{B}_1,\dots,\mathcal{B}_n)$ is inconsistent. 
Note that in general $\mathcal{B}_{k-1}\cap U_\varepsilon$ has more connected components than $C_i^+$ and $C_i^-$, but considering these two components is sufficient for proving the inconsistency condition.

Since $T^n$ is a Lie group and $g_f$ is left-invariant, it follows from Corollary \ref{CorLieGroupTrivCov} that
$$\GC(T^n,g_f)=n+1.$$
\end{example}

Next we will compute the geodesic complexity of arbitrary two-dimensional flat tori. The reader should note that in general, the geodesic complexity of $(T^2,g)$ will vary with the metric $g$, see Example \ref{ExampleGC}.(2). For arbitrary flat tori of higher dimensions, the cut loci of points are not as well-understood as in the two-dimensional case. While it might be possible to extend our result to flat tori of higher dimensions, we are not aware of any systematic study of cut loci of flat higher-dimensional tori in the literature.

For $p,q \in \RR^2$ we let $[p,q]=\{(1-t)p+t q \in \RR^2 \ | \ t \in [0,1]\}$ denote the line segment from $p$ to $q$.
\begin{theorem}
Let $g$ be an arbitrary flat metric on $T^2$. Then $\GC(T^2,g)=3$.
\end{theorem}
\begin{proof}
By elementary Riemannian geometry, $(T^2,g)$ is isometric to $T^2$ with a quotient metric induced by the standard metric on $\RR^2$ and a projection $\pi:\RR^2\to \RR^2/\Gamma=T^2$, where $\Gamma\subset \RR^2$ is a lattice. We thus assume that $g$ itself is such a quotient metric. 
Put $o := \pi(0,0)$. We are going to describe $\Cuttil_o(T^2)$ following \cite[p.108]{GHL}. The case that $\Gamma$ is generated by two orthogonal vectors is covered in Example \ref{ExampleFlatTorus}, so we assume in the following that $\Gamma$ is generated by two vectors $a_1,a_2 \in \RR^2$, such that the angle between $a_1$ and $a_2$ is acute.

If we identify $T_oM$ with $\RR^2$, then $\Cuttil_o(M)$ is given by a hexagon whose construction we will describe next. Consider the perpendicular bisectors of the following line segments:
$$[0,a_1], \quad [0,a_2], \quad [0,-a_1], \quad [0,-a_2], \quad [0,a_1-a_2], \quad [0,a_2-a_1].$$
These perpendicular bisectors enclose a hexagon in $\RR^2$, see Figure \ref{FigureFlatTorus}. The tangent cut locus $\Cuttil_o(M)$ consists of the boundary curve of the hexagon, while the domain of injectivity of $\exp_o$ is given by the interior of the hexagon. Let the segments and the corner points of the hexagon be labelled as in Figure 1. Then there are $p,q \in M$ with $p \neq q$, such that $p=\exp_o(p_1)=\exp_o(p_2)=\exp_o(p_3)$ and $q=\exp_o(q_1)=\exp_o(q_2)=\exp_o(q_3)$. 
\begin{figure}[h]
\centering
\includegraphics[scale=0.6]{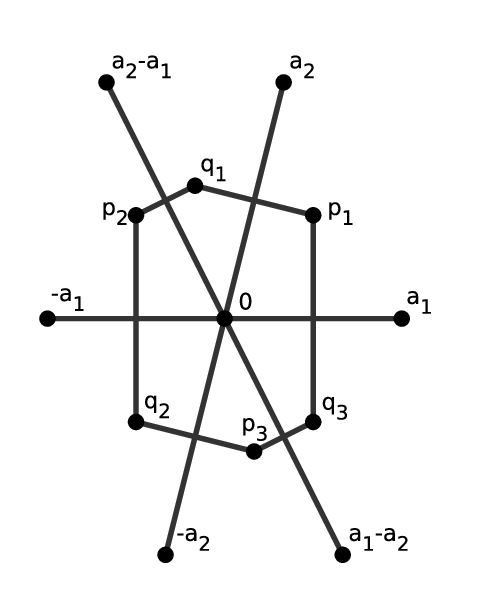}
\caption{Tangent cut loci of flat two-dimensional tori}
\label{FigureFlatTorus}
\end{figure}

For $x,y \in \RR^2$ we put $\llbracket x,y\rrbracket := [x,y] \setminus \{x,y\}$. 
With $p$ and $q$ as above, the set $\Cut_o(M) \setminus \{p,q\}$ has three connected components:
\begin{align*}
&A_1 := \exp_o(\llbracket p_1,q_1 \rrbracket) = \exp_o(\llbracket q_2,p_3 \rrbracket),  \quad A_2:= \exp_o(\llbracket q_1,p_2 \rrbracket)=\exp_o(\llbracket p_3,q_3 \rrbracket), \\
&A_3 := \exp_o(\llbracket p_2,q_2 \rrbracket)=\exp_o(\llbracket q_3,p_1 \rrbracket).
\end{align*}
 More precisely, $\exp_o$ maps both $\llbracket p_1,q_1 \rrbracket$ and $\llbracket q_2,p_3 \rrbracket$ homeomorphically onto $A_1$, both $\llbracket q_1,p_2 \rrbracket$ and $\llbracket p_3,q_3 \rrbracket$ homeomorphically onto $A_2$ and both $\llbracket p_2,q_2 \rrbracket$ and $\llbracket q_3,p_1 \rrbracket$ homeomorphically onto $A_3$.

Let $S_2:= \{p,q\}$ and $S_1:= A_1 \cup A_2 \cup A_3$. By construction, $(S_1,S_2)$ is a trivially covered stratification of $\Cut_o(T_2,g)$. We want to show that $(S_1,S_2)$ is inconsistent as well. Let $K\subset T_oT^2$ denote the union of $\Cuttil_o(T^2,g)$ with the domain of injectivity of $\exp_o$ and let $\exp_K:K \to T^2$ be the restriction of $\exp_o$ to $K$. This is again a special case of the map defined in \eqref{EqexpK}. Let $U\subset T^2$ be an open neighborhood of $p$ and put $Z_i := A_i \cap U$ for all $i \in \{1,2,3\}$. If $U$ is chosen sufficiently small, then by the above description of $\Cuttil_o(T^2,g)$, there are $x_1 \in \llbracket p_1,q_1\rrbracket$ and $x_1' \in \llbracket q_2,p_3\rrbracket$, such that 
$$\exp_{K}^{-1}(Z_1) = \llbracket x_1',p_3\rrbracket \cup \llbracket p_1,x_1\rrbracket .$$
Analogously, one shows that there are 
$$x_2 \in \llbracket q_1,p_2\rrbracket, \quad x_2' \in \llbracket p_3,q_3\rrbracket, \quad x_3 \in \llbracket p_2,q_2\rrbracket \quad \text{and} \quad x_3' \in \llbracket q_3,p_1\rrbracket,$$ such that
$$\exp_K^{-1}(Z_2) = \llbracket x_2,p_2\rrbracket \cup \llbracket p_3,x_2'\rrbracket, \qquad \exp_K^{-1}(Z_3)= \llbracket x_3',p_1\rrbracket \cup \llbracket p_2,x_3 \rrbracket. $$
Since $\exp_K^{-1}(\{p\})=\{p_1,p_2,p_3\}$, this shows that 
\begin{align*}
&\exp_K^{-1}(\{p\})\cap \overline{\exp_K^{-1}(Z_1)}= \{p_2,p_3\}, \qquad \exp_K^{-1}(\{p\})\cap \overline{\exp_K^{-1}(Z_2)}=\{p_1,p_3\}, \\
&\exp_K^{-1}(\{p\})\cap \overline{\exp_K^{-1}(Z_3)}= \{p_1,p_2\}.
\end{align*}
Consequently, 
$$\Cuttil_o(T^2) \cap \exp_o^{-1}(\{p\}) \cap \bigcap_{i=1}^3 \overline{\exp_K^{-1}(Z_i)} =\emptyset,$$
 which shows that $(S_1,S_2)$ satisfies the inconsistency condition at $p$. In complete analogy, one shows that the condition is satisfied at $q$ as well, implying that $(S_1,S_2)$ is inconsistent. Since $g$ is by construction left-invariant, it follows from Corollary \ref{CorLieGroupTrivCov} that $\GC(T^2,g)=3$.
\end{proof}

\subsection{Geodesic complexity of Berger spheres} \label{sec7-2}
In this subsection we consider a class of homogeneous Riemannian manifolds whose geodesic complexity can be computed explicitly without making use of the upper and lower bounds we previously studied. In \cite{BergerSpheres}, M. Berger has constructed a one-parameter family of homogeneous metrics $g_{\alpha}$,  $0<\alpha\leq\frac{\pi}{2}$, on the three-dimensional sphere $S^3$, whose cut loci have been described by Sakai in \cite{SakaiBerger}. 

In the following, we will first recall a particularly interesting class of homogeneous Riemannian manifolds, namely naturally reductive spaces. Berger spheres are special cases of them and we will outline the construction of Berger's metrics following \cite{SakaiBerger}. 


Given a Lie group $G$, we always let $e \in G$ denote its unit element. 
Let $\gg$ denote the Lie algebra of $G$ and assume that $H$ is a closed subgroup of $G$.
Then the Lie algebra $\hh$ of the Lie group $H$ is a Lie subalgebra of $\gg$.
If there is an $\Ad_H$-invariant subspace $\mm$ of the Lie algebra $\gg$ which is complementary to $\hh$ then there is a bijective correspondence between $\Ad_H$-invariant inner products on $\mm$ and $G$-invariant metrics on the homogeneous space $G/H$.
See \cite[Proposition 11.22.(2)]{ONeill} for details.

\begin{definition}[{\cite[p. 317]{ONeill}}]
Let $G$ be a Lie group with a closed subgroup $H$. Let $\gg$ be the Lie algebra of $G$ and $\hh$ be the Lie algebra of $H$. 
Assume that there is a subspace $\mm\subset \gg$ which is complementary to $\hh$ and such that $\Ad_H(\mm)\subset \mm$, where $\Ad_H$ denotes the adjoint representation of $H$.
Suppose there is an $\Ad_H$-invariant inner product $\langle \cdot,\cdot\rangle$ on $\mm$ such that
$$  \langle [X,Y]_{\mm},Z \rangle = \langle X,[Y,Z]_{\mm}\rangle  $$ for all $X,Y,Z\in \mm$ where the subscript $\mm$ of an element of $\gg$ denotes its projection onto $\mm$.
Then $G/H$ together with the $G$-invariant Riemannian metric corresponding to this inner product is called a \emph{naturally reductive space}.
\end{definition}

\begin{example}
All symmetric spaces are examples of naturally reductive spaces as discussed in \cite[p.317]{ONeill}.
The real Stiefel manifolds $V_k(\RR^n)$ for $n\geq 4$ and $2\leq k \leq n-2$ are examples of naturally reductive spaces which are not symmetric spaces, see \cite[p. 748]{GalQuaint}.
\end{example}

For our purposes, the crucial property of naturally reductive spaces is the observation made in the following proposition. We refer to \cite[Proposition 11.25]{ONeill} for its proof.

\begin{prop}
\label{PropGeodNatRed}
Let $G$ be a Lie group and $H \subset G$ be a closed subgroup. If $M=G/H$ is a naturally reductive space and $\pi:G \to M$ is the projection, then the geodesics starting at $o=\pi(e)$ are precisely the curves of the form $\gamma(t) = \pi(\exp(t\xi))$ for $\xi \in \mm$, where $\exp: \gg \to G$ is the Lie group exponential of $G$.
\end{prop}

We proceed by constructing Berger spheres as naturally reductive spaces following the exposition of \cite{SakaiBerger}. Let $G=SU(2)\times \RR$ and let $\gg = \mathfrak{su}(2)\oplus \RR$ be its Lie algebra. We consider the $\Ad_G$-invariant inner product on $\gg$ given by 
$$ \langle (A,x),(B,y) \rangle = -\frac{1}{2}\mathrm{Tr}(AB) + xy  \qquad \forall (A,x),(B,y) \in \gg. $$
For $\alpha \in (0,\frac{\pi}{2}]$ we define a linear subspace of $\gg$ as
$$ \hh_{\alpha} = \left\{ \left( \begin{pmatrix} i \ell\cos\alpha &0 \\ 0& -i\ell\cos\alpha \end{pmatrix},\ell\sin\alpha  \right) \in \gg\, \middle| \,\ell\in\RR\right\} .$$
Consider the closed subgroup $H_\alpha \subset G$, $H_\alpha = \exp(\hh_{\alpha})$, where $\exp$ again denotes the Lie group exponential of $G$. Explicitly, $H_\alpha$ is given as
$$ H_{\alpha} = \left\{ \left( \begin{pmatrix} e^{i\ell \cos\alpha} &0\\ 0& e^{-i\ell \cos\alpha} \end{pmatrix} , \ell \sin\alpha \right)\, \middle| \, \ell\in\RR\right\}  .$$
One checks that $G/H_\alpha$ is diffeomorphic to $S^3$.
The orthogonal complement to $\hh_{\alpha}$ in $\gg$ with respect to $\langle \cdot, \cdot \rangle$ is the space
$$  \mm_{\alpha} = \left\{ \left(\begin{pmatrix} i r \sin\alpha & a+ib \\ -a+ib & -i r \sin\alpha  \end{pmatrix},-r\cos\alpha \right) \in \gg \,\middle|\, a,b,r\in \RR \right\} . $$
A direct computation shows that $\mm_{\alpha}$ is $\Ad_{H_{\alpha}}$-invariant.
The restriction of the inner product $\langle \cdot,\cdot\rangle$ to $\mm_{\alpha}\times\mm_{\alpha}$ defines an $\Ad_{H_{\alpha}}$-invariant inner product on $\mm_{\alpha}$.
We equip the homogeneous space $G/H_{\alpha}$ with the $G$-invariant metric that is defined by this inner product and the abovementioned correspondence between $G$-invariant Riemannian metrics on $G/H_{\alpha}$ and $\Ad_{H_{\alpha}}$-invariant inner products on $\mm_{\alpha}$.

Since $\mm_{\alpha}\perp \hh_{\alpha}$ by construction, the space $M_{\alpha} = G/H_{\alpha}$ equipped with the described homogeneous metric is a naturally reductive space, see \cite[Proposition 23.29]{GalQuaint}.
Thus by Proposition \ref{PropGeodNatRed} the geodesics in $M_\alpha$ emanating from $o$ are precisely the images of the one-parameter groups in $G$ under $\pi$ of elements of $\mm_{\alpha}$.
For $\alpha = \frac{\pi}{2}$, one further observes that $G/H_\alpha$ is isometric to the round sphere $S^3$ of sectional curvature one, see \cite[p. 77]{MainSchmidt}.

The following observation gives us a strong upper bound on $\GC(M_\alpha)$. 
We refer to \cite[Section 3]{MainSchmidt} for its proof.

\begin{prop}
\label{PropBergerLeftIn}
For each $\alpha \in (0,\frac{\pi}2)$ the Berger sphere $M_{\alpha}$ is isometric to $SU(2)$ equipped with a left-invariant metric.
\end{prop}

Combining Proposition \ref{PropBergerLeftIn} with Corollary \ref{CorLieGroup} yields
\begin{equation}
\label{EqBerger}
\GC(M_\alpha) \leq \GC_o(\Cut_o(M_\alpha))+1, 
\end{equation}
where $o= \pi(1)$. To compute $\GC_o(\Cut_o(M_\alpha))$, we will outline the results from \cite{SakaiBerger} about the cut loci of $M_\alpha$. For $\alpha= \frac\pi{2}$, we already know that $\GC(M_{\frac{\pi}{2}})=2$, see Example \ref{ExampleGC}.(1). Thus, in the following we fix $\alpha \in (0,\frac{\pi}{2})$.

Let $S\subset \mm_\alpha$ denote the unit sphere in $\mm_\alpha$ with respect to the norm induced by $\langle \cdot, \cdot\rangle$ and let $D\pi_e:\gg \to T_oM_\alpha$ denote the differential of $\pi$ in the unit $e \in G$. Consider the isometric isomorphism of vector spaces
$$\varphi:= D\pi_e|_{\mm_\alpha}: \mm_\alpha \stackrel\cong\longrightarrow T_oM_\alpha.$$  Then $\varphi$ maps $S$ to the unit sphere in $T_oM_\alpha$. Let $\psi: T_oM_\alpha \to T_oM_\alpha$ be the radial homeomorphism which maps the unit sphere homeomorphically onto the tangent cut locus $\Cuttil_o(M_\alpha)$ of $o$ in $M_\alpha$. Then the map $F: S \to \Cuttil_o(M_\alpha)$, $F:= \psi \circ \varphi|_{S}$ is a homeomorphism. We consider $e_1,e_2 \in \mm_\alpha$ given by 
$$  e_1 = \left(\begin{pmatrix} 0&1\\-1&0 \end{pmatrix},0  \right),\qquad  e_2 = \left(\begin{pmatrix} 0&i\\i&0 \end{pmatrix},0  \right). $$
Let $U := \mathrm{span}_{\RR}(\{e_1,e_2\}) $ and let $r_U: \mm_\alpha \to \mm_\alpha$ denote the reflection through $U$.  Let $\mm_1$ and $\mm_2$ denote the two components of $\mm_\alpha \setminus U$ and put $D_i :=S\cap \mm_i$ for $i \in \{1,2\}$. By the results of \cite[p. 151]{SakaiBerger}: 
\begin{itemize}
	\item $\Cut_o(M_\alpha) = (\exp_o \circ F)(S)$.
	       \item $\exp_o \circ F|_{\overline{D}_1}$ and $\exp_o \circ F|_{ \overline{ D}_2 }$ are injective.
	       \item $(\exp_o  \circ F)(v) = (\exp_o \circ F)(r_U(v))$ for all $v \in D_1 \cup D_2$.
\end{itemize}
Hence, the map $\exp_o\circ F|_{\overline{D}_1}$ is a bijective continuous map from a closed disk onto the cut locus $\Cut_o(M_{\alpha})$.
Since the disk is compact and $\Cut_o(M_\alpha)$ is a Hausdorff space, this shows that $\Cut_o(M_\alpha)$ is homeomorphic to a closed disk. Moreover, $\exp_o\circ F|_{ \overline{D}_i }:\overline{D}_i \to M_\alpha$ is an embedding of $\overline {D}_i$ onto $\Cut_o(M_\alpha)$ for $i \in \{1,2\}$.

\begin{theorem}
\label{TheoremGCBerger}
For all $\alpha \in (0,\frac\pi{2}]$, it holds that $\GC(M_\alpha) = 2.$
\end{theorem}
\begin{proof}
For $\alpha = \frac{\pi}{2}$, i.e. the case of a round metric, this is observed in \cite[Proposition 4.1]{RecioMitter}, so we will only consider the case of $\alpha \in (0,\frac{\pi}{2})$. In the notation from above, we put $E:= F(\overline{D}_1)$ and let $f: \Cut_o(M_\alpha) \to E$, $f := (\exp_o\circ F|_E)^{-1}$. Define
$$s:\{o\} \times \Cut_o(M_\alpha) \to GM_\alpha, \qquad (s(o,q))(t) = \exp_o(t\cdot f(q)) \quad \forall t \in [0,1].$$
By Proposition \ref{PropGeodNatRed}, the map $s$ is a continuous geodesic motion planner, which shows $$\GC_o(\Cut_o(M_\alpha))=1$$ and thus  $\GC(M_\alpha)\leq 2$ by \eqref{EqBerger}. Since $\GC(M_\alpha) \geq \TC(S^3)=2$, this shows the claim.
\end{proof}
\begin{remark}
\begin{enumerate}
    \item 
As Recio-Mitter has shown in \cite[Example 2.4]{RecioMitter}, there exists a Riemannian metric $g_m$ on $S^3$ for which $\GC(S^3,g_m) =3$. This shows that also in the case of $S^3$, the value of $\GC$ depends on the chosen metric.
\item The cut locus of a point in the Berger sphere $M_{\alpha}$ for $0<\alpha<\frac{\pi}{2}$ is a closed disk.
It therefore seems tempting to determine the geodesic complexity of $M_{\alpha}$ via a stratification of this cut locus similarly to what we have done in previous sections.
More precisely, an obvious stratification of a closed disk is given by taking one stratum as its interior and another stratum as its boundary.
However, this is not an inconsistent stratification as in Definition \ref{DefInconsistent} since we would then obtain $\GC(M_{\alpha})\geq 3$, whereas we have shown that $\GC(M_{\alpha}) = 2$.
\item As this example is particularly instructive, we want to sketch briefly how to show directly that the stratification from the previous paragraph is not inconsistent.
Let $K\subset T_o M_{\alpha}$ be the union of the injectivity domain with the tangent cut locus $\Cuttil_o(M_{\alpha})$.
Using the same notation as in the exposition above, put 
$$ \qquad \quad  S_1 := (\exp_o\circ F)(\mathring{D}_1) \subset \Cut_o(M_{\alpha}) \qquad \text{and}\qquad S_2 := (\exp_o\circ F)(\partial D_1)\subset \Cut_o(M_{\alpha}).     $$
Under the identification of $\Cut_o(M_{\alpha})$ and a closed disk, this is the decomposition from part (2) of this remark. Evidently, this is a stratification in the sense of Definition \ref{DefStrat}.

Let $x\in S_2$ and let $U\subset M_{\alpha}$ be a neighborhood of $x$.
For sufficiently small $U$, the intersection $U\cap S_1$ has only one connected component which we call $Z$.
We claim that
$$ \Cuttil_o(M_{\alpha})\cap \exp_o^{-1}(\{x\}) \cap \overline{\exp_K^{-1}(Z)} \neq \emptyset .    $$
By the above discussion of $\Cut_o(M_{\alpha})$, the intersection $\Cuttil_o(M_{\alpha}) \cap \exp_o^{-1}(\{x\})$ consists of a single point $v\in T_oM_\alpha$.
By choosing a sequence in $Z$ converging to $x$ and recalling that $\exp_o\circ F|_{\overline{D}_i}$ is a homeomorphism for $i\in \{1,2\}$, we see that
$$  v\in \overline{\exp_K^{-1}(Z)}.    $$
This shows that
$$ \Cuttil_o(M_{\alpha})\cap \exp_o^{-1}(\{x\}) \cap \overline{\exp_K^{-1}(Z)} = \{v\} \neq \emptyset  . $$
Hence, the stratification $(S_1,S_2)$ of $\Cut_o(M_{\alpha})$ is not inconsistent.

\end{enumerate}
\end{remark}

\section{Explicit upper bounds for symmetric spaces}
\label{sec8}
In \cite[Theorem 5.3]{Sakai1}, Sakai has determined the cut loci of compact simply connected irreducible symmetric spaces. 
He showed that their cut loci always allow for stratifications  for which each stratum is a submanifold.
Since every symmetric space is a Riemannian product of irreducible symmetric spaces, Sakai's results are enough to determine the cut loci of compact simply connected symmetric spaces in general, see our explanations on cut loci of product manifolds in Example \ref{ExampleFlatTorus}. 

In this section, we will first apply the results from Section \ref{sec5} to find an upper bound for the geodesic complexity of a compact simply connected irreducible symmetric space. From Sakai's results, in particular \cite[Proposition 4.10]{Sakai1}, we will further derive estimates on the subspace geodesic complexities of the strata of a cut locus. These numbers appeared on the right-hand side of the inequality in Corollary \ref{CorUpperStrat} and we will show that they can be estimated from above by certain sectional categories. As a result, we will obtain an upper bound for the geodesic complexity of compact simply connected irreducible symmetric spaces given purely in terms of categorical invariants.

We begin by summarizing the main results of \cite{Sakai1}, stated here in the form of \cite[Section 4]{Sakai2}. We assume basic knowledge on symmetric spaces that is provided by textbooks in Riemannian geometry like \cite{Helgason} or \cite{Petersen}.  In the following, we always let $Df_x$ denote the differential of a differentiable map $f$ in the point $x$. \\

Let $M=G/K$ be a compact, simply connected and irreducible symmetric space, where $(G,K)$ is a Riemannian symmetric pair. Explicitly, $G$ is a compact connected Lie group, $K$ is a closed connected Lie subgroup of $G$ and $G$ admits an involutive automorphism $\sigma:G \to G$ whose fixed point set satisfies $(\Fix(\sigma))_0 \subset K \subset \Fix(\sigma)$, where $(\Fix(\sigma))_0$ is the identity component of $\Fix(\sigma)$.

Let $\pi: G \to M$ denote the orbit space projection, let $e \in G$ denote the unit element and put $o := \pi(e) \in M$. Let $\gg$ and $\kk$ denote the Lie algebras of $G$ and $K$, respectively, and let $\mm\subset \gg$ be the $-1$ eigenspace of $D\sigma_e$. Then, since $\kk$ is the $+1$ eigenspace of $D\sigma_e$, there is a vector space decomposition $\gg=\kk\oplus\mm$. Furthermore the restriction $D\pi_e|_{\mm}:\mm \to T_oM$ is a linear isomorphism, see \cite[Theorem IV.3.3]{Helgason}. \\

In the following we give a concise overview over the most important notions related to root systems of symmetric spaces.

\begin{itemize}
	\item Let $\gg_{\CC}$ denote the complexification of $\gg$. By \cite[p. 284]{Helgason}, there exists a Cartan subalgebra $\hh\subset \gg_{\CC}$.
We recall that a \emph{root} $\alpha$ of the Lie algebra $\gg_{\CC}$ is an element of the dual space $\hh^*$ such that there is a non-zero vector $X\in\gg_{\CC}$ satisfying
$$   [H,X] = \alpha(H) X \qquad \text{for all} \ \ H\in\hh .   $$
The set of non-zero roots of the Lie algebra $\gg_{\CC}$ will be called $R$.

\item Let $\aa$ be a maximal abelian subalgebra of $\mm$, which again exists by \cite[p. 284]{Helgason}. We will call $\aa$ a \emph{Cartan subalgebra of $(G,K)$}. 

A root $\alpha\in R$ with $\alpha|_{\aa} \neq 0$ will be called a \emph{root of the symmetric pair $(G,K)$}.
The set of roots of the symmetric pair $(G,K)$ will be denoted by $R(G,K)$.
\item By choosing a certain real subspace $\hh_{\RR}$ of the Cartan subalgebra $\hh$ and defining a lexicographic ordering on $\hh_{\RR}$ one defines an ordering on the set of roots $R$, see \cite[p.173]{Helgason}.
This defines a set of positive roots $R^+\subset R$ of $\gg_{\CC}$.
The set of \emph{positive roots of $(G,K)$} is then defined as
$$  R(G,K)^+ := R^+ \cap R(G,K) .    $$
There is a maximal element of $R(G,K)^+$ with respect to this ordering, which we denote by $\delta$ and call the \emph{highest root of $(G,K)$}.
\item Let $k$ be the rank of the symmetric space $M = G/K$. A \emph{simple root of $(G,K)$} is a positive root $\alpha$ which cannot be written as a sum $\alpha =\beta + \gamma$, where $\beta,\gamma\in R(G,K)^+$.
There are precisely $k$ simple roots and one finds that every positive root can be written as a linear combination of the simple roots with non-negative integer coefficients, see \cite[Theorem VII.2.19]{Helgason}.
Denote the system of simple roots of $(G,K)$ by $\pi(G,K)$. 
\item By virtue of the chosen $\Ad_K$-invariant inner product on $\gg$ we will from now on consider the roots to be vectors in $\aa$ in order to follow \cite[Section 2]{Sakai1}. 
\end{itemize}

Based on this terminology, we next recall Sakai's results on the structure of cut loci of symmetric spaces.
In the case that there are two or more positive roots of $(G,K)$, we define a subset $\D$ of the power set of $\pi(G,K)$ by 
$$\D := \{\Delta \subset \pi(G,K) \ | \ \Delta \neq \emptyset, \ \delta \notin \Delta\}  . $$
If there is only one positive root $\gamma$, which is therefore also the only simple root and also the highest root, define
$$  \D : = \{ \{\gamma \} \}  .  $$
Let $\left< \cdot, \cdot\right>$ denote the chosen $\Ad_K$-invariant inner product on $\gg$ and consider the Weyl chamber of $\pi(G,K)$ that is given by
$$W:= \{X \in \aa \ | \left<\gamma,X \right> >0 \ \ \forall \gamma \in \pi(G,K)\}. $$
 See \cite[Section VII.2]{Helgason} for further details on Weyl chambers and their connection to root systems. 
In case there is more than one positive root, let 
$$S_\Delta := \left\{X \in \overline{W} \ | \ \left<\gamma,X\right> >0 \ \  \forall \gamma \in \Delta, \ \left<\gamma,X\right>=0 \ \ \forall \gamma \in \pi(G,K)\setminus \Delta, \ 2 \left<\delta,X\right>=1  \right\}$$
for each $\Delta \in \D$.
If there is just one positive root $\gamma$, then set
$$  S_{\{\gamma\}} := \{ X\in \aa\, | \, 2\langle \gamma, X \rangle = 1\}   ,  $$
Since $\aa$ is one-dimensional in that case, $S_{\{\gamma\}}$ consists of a single point. 

Let $\exp:\gg\to G$ be the exponential map of $G$ and put 
 $$\Exp:\mm \to M, \qquad \Exp := \pi \circ \exp|_{\mm}.$$
For $\Delta \in \D$ we let
$$\widetilde{\Phi}_\Delta: K \times S_\Delta \to M, \qquad \widetilde{\Phi}_\Delta(k,X) = \Exp(\Ad(k)(X)),$$
and put  $Z_\Delta := \{ k \in K \ | \ \Exp(\Ad(k)(X))=\Exp(X) \ \forall X \in S_\Delta\}$. One checks without difficulties that $Z_\Delta$ is a closed subgroup of $K$. As shown in \cite[Proposition 4.10]{Sakai1}, each $\widetilde{\Phi}_\Delta$ induces a differentiable embedding
$$\Phi_{\Delta}: K/Z_\Delta \times S_\Delta \to M.$$
Put $C_\Delta := \im \Phi_\Delta$ for each $\Delta \in \D$. By \cite[Theorem 5.3]{Sakai1} the cut locus of $M$ at $o$ is then given by 
$$\Cut_o(M) = \bigcup_{\Delta \in \D} C_\Delta $$
and the $C_\Delta$ satisfy
\begin{equation}
C_\Delta \cap C_{\Delta'}=\emptyset \quad \forall \Delta,\Delta'\in \D, \quad \Delta \neq \Delta', \qquad \overline{C}_\Delta= \bigcup_{\Delta' \subset \Delta} C_{\Delta'} \quad \forall \Delta \in \D. \label{EqCDeltaClosure}
\end{equation}
Let $k$ be the rank of $M$. For $i \in \{1,2,\dots,k\}$ we put 
$$\D_i := \{ \Delta \in \D \ | \ \#\Delta=i \} \qquad \text{and}\qquad C_i := \bigcup_{\Delta \in \D_i} C_\Delta.$$ 
It follows from \eqref{EqCDeltaClosure} that $(C_k,C_{k-1},\dots,C_1)$ is a stratification of $\Cut_o(M)$ and that the $C_\Delta$, where $\Delta \in \D_i$, are precisely the connected components of $C_i$. Since $M$ is a homogeneous Riemannian manifold, we thus obtain from Corollary \ref{CorUpperStrat} that
\begin{equation}
\label{EqUpperGCsymmpre}
\GC(M) \leq \secat(\ev_o:\Isom(M)\to M)\cdot \sum_{i=1}^k \max_{\Delta \in \D_i} \GC_o(C_\Delta) + 1.
\end{equation}
It remains to find upper bounds on the numbers $\GC_o(C_\Delta)$. 

\begin{prop}
\label{PropGCCDelta}
For each $\Delta \in \D$ it holds that 
$$\GC_o(C_\Delta) \leq \secat(q_\Delta: K \to K/Z_\Delta),$$
where $q_\Delta$ denotes the orbit space projection. 
\end{prop}
\begin{proof}
Let $r:= \secat(q_\Delta)$. Then by Lemma \ref{LemmaSecatDecomp}, there are pairwise disjoint and locally compact subsets $B_1,\dots,B_r \subset K/Z_\Delta$, such that for each $i \in \{1,2,\dots,r\}$ there is a continuous local section $s_i:B_i \to K$ of $q_\Delta$. Using these $s_i$ we define
$$
\sigma_i: \{o\} \times \Phi_\Delta(B_i \times S_\Delta) \to GM, \qquad (\sigma_i(o,\Phi_\Delta(x,X)))(t) = \Exp(t \cdot \Ad(s_i(x))(X)) ,$$
for every $i \in \{1,2,\dots,r\}$. By construction, each of the $\sigma_i$ is continuous and $\sigma_i(o,\Phi_\Delta(x,X))$ is a geodesic segment for all $(x,X) \in B_i \times S_\Delta$, $i \in \{1,2,\dots,r\}$. Moreover, 
$$(\sigma_i(o,\Phi_\Delta(x,X)))(0)=\Exp(0)=o, \quad (\sigma_i(o,\Phi_\Delta(x,X)))(1) = \Exp(\Ad(s(x))(X))= \Phi_\Delta(x,X),$$
by definition of $\Phi_\Delta$. Thus, the $\sigma_i$ are continuous geodesic motion planners. Since the sets $ \Phi_\Delta(B_1\times S_\Delta),\dots, \Phi_\Delta(B_r \times S_\Delta)$ are pairwise disjoint, locally compact and cover $\Phi(K/Z_\Delta\times S_\Delta)=C_\Delta$, this shows that $\GC_o(C_\Delta) \leq r$.
\end{proof}

Combining Proposition \ref{PropGCCDelta} with \eqref{EqUpperGCsymmpre} yields the following upper bound.

\begin{theorem}
\label{TheoremUpperSymm}
Let $(G,K)$ be a Riemannian symmetric pair and let $M=G/K$ be the corresponding symmetric space. Assume that $M$ is compact, simply connected and irreducible. Then, with $\D_i$ and $Z_\Delta$ given as above, 
$$\GC(M) \leq \secat(\ev_o:\Isom(M) \to M) \cdot \sum_{i=1}^{\rk(M)} \max_{\Delta \in \D_i} \secat(q_\Delta:K \to K/Z_\Delta)+1,$$
where $\rk(M)$ denotes the rank of $M$.
\end{theorem}

\begin{cor}
\label{CorUpperSymmCat}
Let $(G,K)$ be a Riemannian symmetric pair and let $M=G/K$ be the corresponding symmetric space. Assume that $M$ is compact, simply connected and irreducible. Then, with $\D_i$ and $Z_\Delta$ given as above, 
$$\GC(M) \leq \cat(M) \cdot \sum_{i=1}^{\rk(M)} \max_{\Delta \in \D_i} \cat(K/Z_\Delta)+1.$$
\end{cor}
\begin{proof}
This is an immediate consequence of Theorem \ref{TheoremUpperSymm} and \cite[Theorem 18]{SchwarzGenus}.
\end{proof}

We want to conclude by applying the upper bounds to two examples of compact symmetric spaces whose cut loci have already been discussed in the works of Sakai, more precisely in \cite[Example 5.4]{Sakai1} and \cite[Section 4.2]{Sakai3}.

\begin{example}
Consider the complex projective space $\CC P^n = U(n+1)/(U(1)\times U(n))$ with the Fubini-Study metric. This is a compact and simply connected symmetric space of rank one. Its cut locus is studied in detail in \cite[Example 5.4]{Sakai1}.
Let $G=U(n+1)$, let $\gg=\mathfrak{u}(n+1)$ be its Lie algebra and let $K=U(1)\times U(n)$. Let $\gg=\kk\oplus \mm$ denote the decomposition of $\gg$ with respect to the symmetric pair $(G,K)$. By the same methods as in \cite[p. 452]{Helgason}, which treats the Lie algebra of $SU(n)$, one computes that
$$ \mathfrak{k} = \Big\{ \begin{pmatrix} ia& 0 \\ 0 & \xi \end{pmatrix} \,\Big| \, a\in\RR, \xi\in \mathfrak{u}(n) \Big\} \qquad \text{and} \qquad \mathfrak{m} = \Big\{ \begin{pmatrix} 0 & \Bar{u}^T \\ -u & 0 \end{pmatrix} \,\Big|\, u\in \CC^n \Big\}  .$$
Then $\aa= \mathrm{span}_{\RR}(\{H_0\})$ is a Cartan subalgebra of $(G,K)$, where $H_0=(h_{ij})\in \gg$ is given by $$h_{ij} = \begin{cases}
\frac{\pi}{2} & \text{if } (i,j)=(1,n+1), \\
-\frac{\pi}{2} & \text{if } (i,j)=(n+1,1), \\
0 & \text{else}. 
\end{cases}$$
In particular, every system of simple roots of $(G,K)$ consists of a unique element. Let $o$ be the equivalence class of the neutral element of  $G$ in $G/K=\CC P^n$. Then $\Cut_o(M)$ consists of a unique submanifold, given by 
 $$\Cut_o(M)= \{\Exp(\Ad(k)(H_0)) \ | \ k \in U(1) \times U(n) \}.$$ Sakai further showed that $$Z_0 := \{ k \in U(1) \times U(n) \ | \ \Exp(\Ad(k)(H_0))=\Exp(H_0))\}$$ can be identified with $Z_0 =U(1) \times U(n-1)\times U(1).$
Hence, by Proposition \ref{PropGCCDelta},
\begin{equation}
\label{EqGCsecatUn}
\GC_M(\Cut_o(M)) \leq \secat( U(1)\times U(n)\to (U(1)\times U(n)) / (U(1)\times U(n-1)\times U(1)) ) . 
\end{equation}
 One easily checks that the map
 \begin{align*}
 \varphi: (U(1)\times U(n)) / (U(1)\times U(n-1)\times U(1)) &\to U(n)/ (U(n-1)\times U(1))=\CP^{n-1}, \\ \varphi([z,A])&=[A],
  \end{align*}
 where $(z,A) \in U(1)\times U(n)$, is a well-defined homeomorphism. Let $p:U(n) \to \CP^{n-1}$ denote the principal $U(1)$-bundle over the homogeneous space $\CP^{n-1}$. Assume that $s : V\to U(n)$ is a continuous local section of $p$ over a subset $V\subset \CC P^{n-1}$. Then we obtain a continuous local section $\widetilde{s}:V \to U(1) \times U(n)$ of the principal fiber bundle 
 $$U(1)\times U(n)\to (U(1)\times U(n)) / (U(1)\times U(n-1)\times U(1))$$
 by setting $\widetilde{s}(\varphi^{-1}(p)) = (z_0,s(p))$ for $p \in V$, where $z_0 \in U(1)$ is a fixed element.
 This shows that
 $$    \secat( U(1)\times U(n)\to (U(1)\times U(n)) / (U(1)\times U(n-1)\times U(1)) )   \leq \secat(U(n)\to\CC P^{n-1}) . $$
Hence, we derive from \eqref{EqGCsecatUn} that
$$ \GC_o(\Cut_o(M)) \leq \secat(U(n)\to \CC P^{n-1}) \leq \cat(\CC P^{n-1}) = n, $$
where we used \cite[Theorem 18]{SchwarzGenus} for the second inequality. 
The fact that $\cat(\CC P^{n-1})=n$ is shown in \cite[Example 1.51]{CLOT}. Eventually, by Theorem \ref{TheoremGCCut} and the same references,
\begin{align*}
    \GC(\CC P^n )  & \leq  \secat(U(n+1) \to \CC P^n) \GC_o(\Cut_o(M)) +1 \\
    &\leq  \cat(\CC P^n) \cat(\CC P^{n-1} )+ 1 \\
    &= (n+1) n +1 .
\end{align*}
Since $\TC(\CP^n)=2n+1$ as computed in \cite[Lemma 28.1]{FarberSurveyTC}, we derive using Remark \ref{RemarkTC}.(1) that 
$$2n+1 \leq \GC(\CP^n) \leq (n+1)n+1 \qquad \forall n \in \NN.$$
For $n=2$, this shows that $\GC(\CC P^2) \in \{5,6,7\}$.
\end{example}

\begin{example}
Consider the complex Grassmann manifold $G_{2}(\CC^4)=U(4)/(U(2)\times U(2))$. As a quotient of a compact Lie group by a closed subgroup, $G_2(\CC^4)$ is compact. Let $V_2(\CC^4)$ be the corresponding complex Stiefel manifold. As shown in \cite[Example 4.54]{Hatcher}, $V_2(\CC^4)$ is simply connected. Since the fiber $U(2)$ of the fibration $V_2(\CC^4)\to G_2(\CC^4)$ is connected, it follows from the long exact homotopy sequence of this fibration that $G_2(\CC^4)$ is simply connected as well. The cut loci of $G_2(\CC^4)$ are discussed in \cite[Section 4.2]{Sakai3} and \cite[Example 5.5]{Sakai1}. The corresponding decomposition of the Lie algebra $\mathfrak{g}= \mathfrak{u}(4)$ of $G=U(4)$ is given by $\mathfrak{g}=\kk\oplus\mathfrak{m}$ where
$$ \kk = \Big\{ \begin{pmatrix} \alpha&0\\ \ 0 & \beta \end{pmatrix} \,\Big|\, \alpha,\beta\in \mathfrak{u}(2)\Big\} \qquad \text{and} \qquad \mathfrak{m} = \Big\{ \begin{pmatrix} 0& \xi\\ -\bar{\xi}^T & 0 \end{pmatrix} \, \Big| \, \xi \in M_2(\CC) \Big\} . $$
Here, $\kk$ is the Lie algebra of $K=U(2) \times U(2)$. A Cartan subalgebra $\aa\subset \mm$ is spanned by 
$$ e_1 := \frac{1}{2\pi} \begin{pmatrix} 0&0&1&0 \\ 0&0&0&0 \\ -1&0&0&0 \\ 0&0&0&0 \end{pmatrix} \qquad \text{and} \qquad e_2 :=  \frac{1}{2\pi}\begin{pmatrix} 0&0&0&0 \\ 0&0&0&1 \\ 0&0&0&0 \\ 0&-1&0&0 \end{pmatrix}.$$

By \cite[p. 143]{Sakai1}, one can define positive roots and simple roots of $(U(4),U(2)\times U(2))$ in such a way that $2e_1 = \delta$ is the highest root and that a system of simple roots is given by
$$ \pi(G,K) = \{ \gamma_1, \gamma_2\}, \quad \text{where } \gamma_1:=2e_2, \ \gamma_2:=e_1-e_2. $$ 
Thus, in the notation from above, $\D= \{\Delta_0,\Delta_1,\Delta_2\}$, where $\Delta_0=\{\gamma_1,\gamma_2\}$, $\Delta_1=\{\gamma_1\}$ and $\Delta_2 = \{\gamma_2\}$. With $S_i := S_{\Delta_i}$ for $i \in \{0,1,2\}$, one computes that
$$S_0 = \{\pi^2e_1 + \lambda e_2 \in \aa \ | \ \lambda \in (0,\pi^2)\}, \quad S_1 = \{\pi^2e_1\}, \quad S_2 = \{\pi^2(e_1+e_2)\}.$$
We further put $Z_i:= Z_{\Delta_i}$ for each $i$. 
By computing the corresponding matrix exponentials, we obtain
$$ Z_1 = \Big\{ \mathrm{diag}(a,b,c,d)\in U(2) \times U(2) \ \Big|\ a,b,c,d\in U(1)\Big\} \cong U(1)^4 .$$
Since $U(2)/(U(1)\times U(1))$ is diffeomorphic to $\CP^1\cong S^2$, it follows that $K/Z_1 \cong S^2 \times S^2$. Hence, $\cat(K/Z_1)=\cat(S^2 \times S^2)\leq 3$ by the product inequality for $\cat$, see \cite[Theorem 1.37]{CLOT}. 
One further computes by matrix exponentials that $Z_2 = K$, so $K/Z_2$ consists of a single point, which yields $\cat(K/Z_2)=1$. 
By \cite[Lemma 4.9]{Sakai1}, for every fixed $X \in S_0$ we obtain $Z_0 = \{k \in K \ | \ \Exp(\Ad(k)(X))=\Exp(X)\}$. 
We choose $X=\pi^2e_1+\frac{\pi^2}{2}e_2$ and claim that
$$ Z_{0} = \left\{ \mathrm{diag}(a,b,c,b) \in U(2) \times U(2) \  \middle|\ a,b,c\in U(1)\right\}.   $$
To see this, we compute that
$$ \exp(X) = \begin{pmatrix} 0&0&1&0 \\ 0 &\frac{1}{\sqrt{2}} & 0 &\frac{1}{\sqrt{2}} \\ -1 &0&0&0 \\ 0&\frac{-1}{\sqrt{2}}&0&\frac{1}{\sqrt{2}}  \end{pmatrix} .      $$
The condition $\Exp(\Ad_k(X)) = \Exp(X)$ is equivalent to 
$  \exp(-X)k\exp(X)\in K . $
One then checks by an explicit computation that $k\in K$ satisfies this condition if and only if $$ k = \mathrm{diag}(a,b,c,b), \qquad \text{with}\,\,a,b,c\in U(1).   $$ 
Hence, $K\to K/Z_0$ is a bundle with typical fiber $U(1)^3$, where an inclusion of the fiber is given by $$f:U(1)^3 \to U(2) \times U(2), \qquad f(a,b,c)=\diag(a,b,c,b).$$ We want to show that $K/Z_0$ is simply connected. By the long exact sequence of homotopy groups of that bundle, it suffices to show that $f_*:\pi_1(U(1)^3)\to \pi_1(U(2)^2)$ is surjective. Let $\gamma:[0,1]\to U(1)$, $\gamma(t)=e^{2\pi i t}$. We observe that $\pi_1(U(1)^3)\cong \ZZ^3$. A set of generators of $\pi_1(U(1)^3)$ is given by the homotopy classes of the loops $$\gamma_1,\gamma_2,\gamma_3:[0,1] \to U(1)^3, \quad \gamma_1(t)=(\gamma(t),1,1), \quad\gamma_2(t)=(1,\gamma(t),1), \quad \gamma_3(t)=(1,1,\gamma(t)).$$ We further observe that $\pi_1(U(2)^2)\cong \ZZ^2$, where a set of generators is given by the homotopy classes of $$\beta_1,\beta_2:[0,1] \to U(2)\times U(2), \qquad \beta_1(t)=\diag(\gamma(t),1,1,1), \qquad \beta_2(t)=\diag(1,1,\gamma(t),1).$$ Here, we used \cite[Example VII.8.1]{Bredon}. One immediately sees that $f\circ \gamma_1 = \beta_1$ and $f \circ \gamma_3=\beta_2$. This shows that the image $f_*$ contains a set of generators, hence $f_*$ is surjective. Thus, $\pi_1(K/Z_0)$ is the trivial group, which implies by \cite[Theorem 1.50]{CLOT} that 
$$ \cat(K/Z_0) \leq \frac{\dim (K/Z_0)}{2}+1 = \frac52 +1 = \frac{7}{2}.$$ 
Since $\cat$ is integer-valued, we obtain $\cat(K/Z_0) \leq 3$. To employ Corollary \ref{CorUpperSymmCat}, we still need to estimate $\cat(G_2(\CC^4))$ from above.  Another use of \cite[Theorem 1.50]{CLOT} shows that $$\cat (G_2(\CC^4))\leq \frac12\dim (G_2(\CC^4))+1=5.$$ Inserting the results of our computations into Corollary \ref{CorUpperSymmCat}, we derive 
\begin{align*}
\GC(G_2(\CC^4)) &\leq \cat(G_2(\CC^4)) ( \cat(K/Z_0) + \max \{\cat(K/Z_1),\cat(K/Z_2)\})+1 \\
&\leq 5(3+3)+1= 31.
\end{align*}
By \cite[Lemma 28.1]{FarberSurveyTC} it further holds that $\TC(G_2(\CC^4))=\dim (G_2(\CC^4))+1=9$. Thus, by the previous inequality and Remark \ref{RemarkTC}, we obtain
$$9 \leq \GC(G_2(\CC^4))\leq 31.$$
\end{example}
\bibliography{GC}
 \bibliographystyle{amsalpha}
 
\end{document}